\documentclass[english,11pt]{article}

\usepackage[english]{babel}
\usepackage{amsmath,amsthm,amssymb,enumerate}
\usepackage[latin1]{inputenc}
\usepackage[T1]{fontenc}

\usepackage{stix}
\usepackage{relsize}
\usepackage{tikz}
\usepackage{times}
\usepackage{tkz-graph}
\usetikzlibrary{mindmap,backgrounds}
\tikzstyle{vertex}=[circle, draw, inner sep=0pt, minimum size=6pt]
\newcommand{\vertex}{\node[vertex]}
\usetikzlibrary{matrix,arrows,calc,decorations.pathreplacing}
\usepackage{enumitem}
\usepackage{tcolorbox}
\usepackage[all]{xy} 

\usepackage{psfig,labelfig}
\usepackage{graphicx}
\usepackage[dvips]{epsfig}

\usepackage{hyperref}

\usepackage[center, labelfont={normalfont,bf},format=hang,,singlelinecheck=false]{caption}

\usepackage[colorinlistoftodos,prependcaption]{todonotes}

\newtheorem{thm}{Theorem}[section]
\newtheorem{con}[thm]{Conjecture}
\newtheorem{cor}[thm]{Corollary}
\newtheorem{prop}[thm]{Proposition}

\newtheorem{note}[thm]{Note}
\newtheorem{defi}[thm]{Definition}
\newtheorem{lem}[thm]{Lemma}
\newtheorem{ex}[thm]{Example}
\newtheorem{rem}[thm]{Remark}

\newcommand{\OriSquare}[5]
{
\begin{xy}
(0,0)="Pos";
"Pos"+(1,0) **@{-}; ?*h!U!/^2pt/{\text{\scriptsize #5}}, 
"Pos"+(1,1) **@{-}; ?*h!L!/^1pt/{\text{\scriptsize #2}},
"Pos"+(0,1) **@{-}; ?*h!D!/^1pt/{\text{\scriptsize #3}},
"Pos" **@{-}; ?*h!R!/^1pt/{\text{\scriptsize #4}},
"Pos"+(0.5,0.5) *h{\text{\scriptsize #1}};
\end{xy}
}

\newcommand{\R}{{\mathbb R}}

\newcommand{\Z}{{\mathbb Z}}
\newcommand{\N}{{\mathbb N}}

\newcommand{\RR}{\mathbb{R}}
\newcommand{\ZZ}{\mathbb{Z}}

\newcommand{\HHH}{\mathcal{H}}

\newcommand{\hex}{\mathop{Hex}}

\newcommand{\fs}{\footnotesize}

\newcommand{\GIG}{\mathcal{G}_{\Gamma}}
\newcommand{\blackcircle}[2][black,fill=black]{\tikz[baseline=-0.5ex]\draw[#1,radius=#2] (0,0) circle ;}%
\newcommand{\whitecircle}[2][black]{\tikz[baseline=-0.5ex]\draw[#1,radius=#2] (0,0) circle ;}%

\DeclareMathOperator{\sy}{sys}

\DeclareMathOperator{\sr}{SR}
\DeclareMathOperator{\syh}{sys_h}
\DeclareMathOperator{\srh}{SR_h}
\DeclareMathOperator{\dist}{dist}

\DeclareMathOperator{\area}{area}

\DeclareMathOperator{\cat}{CAT}
\DeclareMathOperator{\iso}{Isom}

\DeclareMathOperator{\SL}{SL}

\newcommand{\sm}{\scriptscriptstyle}

\newcommand{\smax}{s_{\fs max}}

\newcommand*\circled[1]{\tikz[baseline=(char.base)]{%
    \node[shape=circle,draw,inner sep=1.5pt] (char) {#1};}}

\title{Systolic geometry of translation surfaces}

\date{}

\author{Tobias Columbus, Frank Herrlich, Bjoern Muetzel and Gabriela Weitze-Schmith\"usen}

\begin{document}
\maketitle

\begin{abstract}
  In this paper we investigate the systolic landscape of translation surfaces for fixed genus and fixed angles of their cone points. We furthermore study how the systoles of a translation surface relate to the systoles of its graph of saddle connections. This allows us to develop an algorithm to compute the systolic ratio of origamis in the stratum  $\mathcal{H}(1,1)$. We compute the maximal systolic ratio of all origamis in $\mathcal{H}(1,1)$ with up to 67 squares. These computations support a conjecture of Judge and Parlier about the maximal systolic ratio in $\mathcal{H}(1,1)$.
\end{abstract}

\section{Introduction}
\textit{Translation surfaces} are closed flat surfaces $S$ with singularities or cone points $(p_i)_{i=1 ,\ldots, n}$ with cone angle $2\pi \cdot (k_i+1)$ at $p_i$, where $k_i \in \N$.
They have been intensively studied from the perspective of dynamical systems, algebraic geometry and geometric group
theory for about 30 years now, see for example \cite{fm}, \cite{wr} and \cite{zo} for overview articles.\\
It follows from the Euler characteristic of the surface that 
\begin{equation}
      \sum_{i=1}^n k_i = 2g-2.
\label{eq:Euler}
\end{equation}
In this article we restrict to translation surfaces of genus $g \geq 2$.
It follows from the formula above that the space of translation surfaces of genus $g \geq 2$ can be subdivided into admissible \textit{strata} $\mathcal{H}(k_1 ,\ldots, k_n)$, such that the $(k_i)_{i=1 ,\ldots, n}$ satisfy Equation (\ref{eq:Euler}). The moduli space $\mathcal{M}^{tr}_g$ of translation surfaces of genus $g \geq 2$ has real dimension $8g-6$. The moduli space $\mathcal{M}	(\mathcal{H}(k_1 ,\ldots, k_n))$ of the surfaces in the stratum $\mathcal{H}(k_1 ,\ldots, k_n)$ is an orbifold of real dimension $4g+2n-2$ (see \cite{kz}):
\[
           \dim(\mathcal{M}^{tr}_g)= 8g-6   \text{ \ and \ }\dim( \mathcal{M}	(\mathcal{H}(k_1 ,\ldots, k_n)))= 4g+2n-2.
\]
We note that the stratum $\mathcal{H}(1,1 ,\ldots, 1)$ is the largest and the only one with full dimension.\\ 
A \textit{systole} of a translation surface $S$ is a shortest simple closed geodesic. We denote by $\sy(S)$ its length. The aim of this article is to investigate the systolic landscape of these surfaces. The leitmotif for the proofs of the geometric inequalities of Section 2 and 3 is that translation surfaces are $\cat(0)$ spaces. Therefore these surfaces share many properties of surfaces with non-positive curvature and especially the well-studied hyperbolic surfaces (see, for example, \cite{ak},  \cite{ba}, \cite{bu1}, \cite{bu2},  \cite{ge}, \cite{pa1}, \cite{ra} \cite{sc1}, \cite{sc2}  and \cite{sc3}). The main idea of these sections is therefore to tweak and carry over the corresponding results for compact hyperbolic surfaces to the realm of translation surfaces. One property of translation surfaces is that there is always a systole that passes through a cone point. We furthermore show in Section 2 that large systoles always have large collars in the systolic collar lemma (see Lemma \ref{thm:cyl_sys}).\\
Normalizing by the area $\area(S)$ of $S$ we obtain $\sr(S) = \frac{\sy(S)^2}{\area(S)}$, the \textit{systolic ratio} of $S$, which is invariant under scaling of $S$. Let 
\[
     \sr^{tr}(g) = \sup \{ \sr(S) \mid S \text{ translation surface of genus } g \}.
\]
be the \textit{supremal systolic ratio in genus} $g$. We also define the \textit{homological systole}, which is a shortest homologically non-trivial loop in $S$. This is a shortest non-contractible loop that does not separate $S$ into two parts. We denote by $\syh(S)$ its length and define $\srh(S) =\frac{\syh(S)^2}{\area(S)}$ as the \textit{homological systolic ratio}. Let

\[
     \sr_h^{tr}(g) = \sup \{ \srh(S) \mid S \text{ translation surface of genus } g  \}
\]
be the \textit{supremal homological systolic ratio in genus} $g$. It follows that for any surface $S$ 
\begin{equation}
      \sy(S) \leq \syh(S), \text{ \ \ hence \ \ }  \sr^{tr}(g) \leq \sr_h^{tr}(g) \leq \frac{(\log(195g)+8)^2}{\pi (g-1)} \text{ \ \ for \ } g \geq 76.
\label{eq:srh_up}      
\end{equation}
Here the upper bound follows from \cite{am1}, Theorem 1.3-3. The inequality stated there is valid for any smooth Riemannian surface. It also applies in the case of translation surfaces, as any translation surface can be approximated by smooth Riemannian surfaces. This means that systolic ratio in genus $g$ can only be of order $\frac{\log^2(g)}{g}$. In the case of the systolic ratio for the genus it is clear that this is indeed a maximum (see Section 3). We call a surface $S_{max}$ \textit{maximal}, if $\sr^{tr}(g)$ is attained in this surface. Maximal surfaces have the following property:

\begin{thm} Let $S_{max}$ be a maximal translation surface of genus $g \geq 1$. Then every simple closed geodesic, that does not run through a cone point is intersected by a systole of $S_{max}$. 
\label{thm:char_Smax}
\end{thm} 
For a fixed stratum $\mathcal{H}(K)$, where $K=(k_1 ,\ldots, k_n)$ we define in a similar fashion
\begin{eqnarray*}
     \sr(\mathcal{H}(K)) = \sup \{ \sr(S) \mid S \text{ translation surface in } \mathcal{H}(K)\}
\end{eqnarray*}
and  $\srh(\mathcal{H}(K))$. 
\\
In this case it is not clear whether this is a maximum or a supremum. The problem is that two or more cone points might merge in a sequence of surfaces in which the systolic ratio goes to the limit.  
\\
Concerning all these invariants surprisingly little is known in the case of translation surfaces. In the case of the systolic ratio of genus $g$ only the case of genus one is clear. In the case of flat tori $\sr^{tr}(1) = \frac{2}{\sqrt{3}}$. In this case the maximal surface is the equilateral torus, that has a hexagonal lattice. In genus two Judge and Parlier conjecture in \cite{jp} that the surface $\hex_2$ obtained by gluing parallel sides of two isometric cyclic hexagons is maximal. The systolic ratio of this surface is 
\begin{equation}
        \sr({\hex}_2) = 0.58404...  .        
\label{eq:sr_hex2}        
\end{equation}
In this article they also show that in the case of the stratum $\mathcal{H}(2g-2)$ the maximum is attained in surfaces $\triangle_g$ of genus $g$ composed of equilateral triangles and that
\begin{equation}
         \sr(\triangle_g)  = \sr(\mathcal{H}(2g-2)) = \frac{4}{\sqrt{3} \cdot (4g-2)}.
\label{eq:sr_H2g-2}         
\end{equation}
Concerning the lower bound of $\sr^{tr}(g)$ we show in this article:

\begin{thm}[Intersystolic inequalities] Let $ \sr^{tr}(g)$ and  $\sr_h^{tr}(g)$ be the supremal systolic ratio and homological systolic ratio in genus $g$. Then 
\[
                \frac{\sr^{tr}(g)}{k}  \leq   \sr^{tr}(k(g-1)+1)   \text{ \ and \ }   \frac{\sr_h^{tr}(g)}{k}  \leq   \sr_h^{tr}(k(g-1)+1).
\]
\label{thm:sr_cyclic}
\end{thm}

From Theorem \ref{thm:sr_cyclic} and Equation \eqref{eq:srh_up} and  \eqref{eq:sr_hex2} we conclude: 

\begin{cor} Let $ \sr^{tr}(g)$ and  $\sr_h^{tr}(g)$ be the supremal systolic ratio and homological systolic ratio in genus $g$, respectively. Then 

\[
                0.584  \leq   \sr^{tr}(2)     \text{ \ \ hence \ \ }     \frac{0.58}{g-1}  \leq   \sr^{tr}(g) \leq \sr_h^{tr}(g) \text{ \ for all \ } g \geq 3.
\]
\label{cor:sr_cyclic}
\end{cor}
To obtain this theorem we construct explicitly cyclic covering surfaces of genus $k(g-1) +1$ for a given surface of genus $g$. The theorem then follows from the fact that the length of a systole does not decrease in a covering surface. As we can also control the stratum of the covering surface a similar theorem for strata is stated in Corollary \ref{thm:sr_cyclic_strata}. Using a simple area argument we also show that 

\begin{thm}[Area estimate] Let $S$ be a translation surface in the stratum $\mathcal{H}(K)$, for $K=(k_1 ,\ldots, k_n)$, such that  
$k_1 \leq k_2 \leq \ldots\leq k_n$. Then
\[
    \sr(\mathcal{H}(K)) \leq \frac{4}{\pi \cdot (k_n+1)}.
\]      
\label{thm:SR_stratum}
\end{thm}
This inequality seems to be useful for large $k_n$. In fact in the case of the stratum $\mathcal{H}(2g-2)$ it is close to the optimal inequality obtained from Equation \eqref{eq:sr_H2g-2} which follows from the results of  Judge/Parlier (cf. above) and Boissy/Geninska (cf. below).
Another important type of curves on translation surfaces are the saddle connections. A \textit{saddle connection} of a translation surface $S$ is a geodesic arc, whose endpoints are cone points, where we allow the case that both endpoints are the same cone point. As there is always a systole that runs through a cone point (see Proposition \ref{thm:sys_tr}), there is always a systole that is a saddle connection  in $\mathcal{H}(2g-2)$. While results about moderate length saddle connections and simple closed geodesics in any stratum were shown earlier in \cite{vo}, recently Boissy and Geninska showed in \cite{bg}, Theorem 3.3.:
\begin{thm}[Boissy, Geninska] Let $S$ be a translation surface in the stratum $\mathcal{H}(K)$, where $K=(k_1 ,\ldots, k_n)$, where $\sum_{i=1}^n k_i = 2g-2$. Then the shortest saddle connection $\delta$ in $S$ satisfies
\[
      \frac{\ell(\delta)^2}{\area(S)}  \leq \frac{2}{\sqrt{3}(2g-2+n)}
\]
The equality is obtained if and only if $S$ is built with equilateral triangles with sides saddle connections of length $\ell(\delta)$. Such surface exists in any connected component of any stratum. 
\label{thm:saddle_nohorse}
\end{thm}
This result therefore extends the result in \cite{jp} about the stratum $\mathcal{H}(2g-2)$. As there is always a systole that passes through a saddle point this inequality also shows that $\sr^{tr}(g)$ is at least of order $\frac{1}{g}$. The result is, however slightly weaker than Corollary  \ref{cor:sr_cyclic}. In terms of inequality \eqref{eq:srh_up} and Corollary \ref{cor:sr_cyclic} the remaining question is if $\sr^{tr}(g)$ and $\sr_h^{tr}(g)$ are of order $\frac{\log(g)^2}{g}$ or of order $\frac{1}{g}$. Our intuition is that the moduli space of translation surfaces is 'large' enough to attain the upper bound that is also attained in the general case. Therefore we conjecture:

\begin{con} Let $ \sr^{tr}(g)$ and  $\sr_h^{tr}(g)$ be the supremal systolic ratio and homological systolic ratio in genus $g$, respectively. Then 
\[
     \sr^{tr}(g) \text{ \ and \ } \sr_h^{tr}(g) \text{ \ are of order \ } \frac{\log(g)^2}{g}.
\] 
\end{con} 

Using similar methods as in the proof of Theorem \ref{thm:SR_stratum} we furthermore show that there always exist a certain number of short saddle connections depending on the degree and number of cone points in the surface. 

\begin{thm}[Short saddle connections] Let $S$ be a translation surface of genus $g$ in the stratum $\mathcal{H}(K)$ where $K=(k_1 ,\ldots, k_n)$, such that $k_1 \leq k_2 \leq \ldots \leq k_n$. Then there exist $\lfloor \frac{n}{2} \rfloor$ saddle connections $(\delta_l)_{l=1 ,\ldots, \lfloor \frac{n}{2} \rfloor}$, such that  
\[
  \frac{\ell(\delta_1)^2}{\area(S)} \leq \frac{4}{\pi(2g-2+n)}  \text{ \ \ and \ \ } \frac{\ell(\delta_l)^2}{\area(S)}  \leq \frac{4}{\pi(2g+n - 2l- \left( \sum_{i=0}^{2l-3}   k_{n-i}  \right) )}, \text{ \ for \ } l \geq 2.
\]      
\label{thm:saddle} 
\end{thm}  
The upper bound for the shortest saddle connection given by $\ell(\delta_1)$ is only slightly weaker than the one provided in Theorem \ref{thm:saddle_nohorse}.
In the largest stratum $\mathcal{H}(1,1,\ldots,1)$ this theorem implies: 

\begin{cor}[Short saddle connections in  $\mathcal{H}(1,1,\ldots,1)$ ] Let $S$ be a translation surface of genus $g$ in the stratum $\mathcal{H}(1,1,\ldots,1)$. Then there exist $ g-1 $ saddle connections, such that
\[
 \frac{\ell(\delta_l)^2}{\area(S)}  \leq \frac{1}{\pi(g -l)}  \text{ \ for \ } l \in \{1,2, \ldots,g-1\}. 
\] 
\label{thm:saddle_111}
\end{cor}  

The systoles of a translation surface $S$ are closely related to the systoles of its graph $\Gamma$ of saddle connections (cf. Section~\ref{section:gos}). We study the exact relations between them and obtain the following result:
\begin{thm}[Relation between systoles in the graph of saddle connections and of the surface] \label{n3thm}
  Let $c$ be a systole in the graph $\Gamma$ and let $\gamma$ be the corresponding closed curve on the surface given as union of saddle connections. If the combinatorial length of $c$ is not 3 then $\gamma$ is a systole of the surface. Furthermore, if the combinatorial length of $c$ is 3 and all angles of $c$ are greater or equal to $\pi$ (cf. Section~\ref{section:gos}) then $\gamma$ is a systole of the surface, too.
\end{thm}
It turns out that in the stratum $\mathcal{H}(1,1)$
the length of the systoles on $S$ actually equals the length of the systoles in $\Gamma$ (cf. Corollary~\ref{h11}), whereas we construct a translation surface in $\mathcal{H}(1,1,1,1)$ for which this is not the case 
(cf. Example~\ref{ex:Exception})
. We finally consider special translation surfaces called \textit{origamis}. They lie dense in each stratum and thus can be used to determine the maximal systolic ratio for a stratum. We  
present an algorithm to compute the systoles of the graph of saddle connections of a given origami surface (cf. Algorithm I and II). We use this in order to compute the maximal systolic ratio of all origamis with at most 67 squares in $\mathcal{H}(1,1)$ (cf. Section~\ref{section:algorithm}). These computations support Conjecture 1.2 of Judge and Parlier in \cite{jp}.\\
This article is structured in the following way. After introducing the necessary tools and definitions in Section 2 we present the results about short geodesics on translation surfaces. Then we provide an interesting property of maximal surfaces in Section 3. In Section 4, we introduce the graph of saddle connections and clarify how its systoles relate to the systoles of the translation surface. Finally, in Section 5 we present the algorithm to compute the length of a systole of the graph of saddle connections for an origami surface. 

\section*{Acknowledgments} We would like to thank Chris Judge for helpful discussions. We thank Pascal Schumann for 
pointing out some errors and typos to us. We would also like to thank the referees for their very helpful comments.
This work contributes to  Project-ID 286237555 - TRR 195 -- by the Deutsche Forschungsgemeinschaft (DFG, German Research Foundation).

\section{Short curves on translation surfaces}
As translation surfaces have singularities, we first give a proper definition of a geodesic. A curve is a map  $\gamma: I \rightarrow S, t \mapsto \gamma(t)$, from an open or closed interval $I \subset \R$ into a translation surface $S$. A \textit{geodesic} is a piece-wise differentiable curve, such that for all $x \in I \backslash \partial I$ there is a neighborhood $U_x$ of $x$, such that $\gamma \mid_{U_x}$ is an isometry.\\
By abuse of notation we denote the image $\gamma(I)$ equally by the letter $\gamma$. Denote by $\ell(\gamma)$ its length.\\
A \textit{geodesic arc} $\gamma_{p,q}$ in $S$ is a geodesic in $S$ with starting point $p$ and endpoint $q$. A \textit{geodesic loop} $\gamma_{p}$ in $S$ is a geodesic arc with a single starting and endpoint $p$.\\
As the cone angles in the cone points are always bigger or equal to $4\pi$, translation surfaces are local $\cat(0)$ spaces (see for example \cite{pa}, Theorem 3.15). It follows that the universal covering space $\tilde{S}$ of a translation surface $S$ is a global $\cat(0)$ space and $\tilde{S}$ is homeomorphic to $\R^2$. There exists a group $G$ of Deck transformations with the following properties:
\[
   S \simeq {\tilde{S} \mod G}, \text{ \ \ \ } G \subset \iso(S)  \text{ \ and \ } G \simeq \pi_1(S),
\]
where $\iso(S)$ denotes the group of isometries of $S$. Furthermore the projection
\[
 pr:\tilde{S} \rightarrow S
\] 
is a local isometry. Denote by $B_r(p) \subset S$ an open disk of radius $r$ around $p \in S$. Let  
\[
U_r(p)=\left\{q \in S \mid \dist(p,q) < r \right\}
\]
be the set of points with distance smaller than $r$ from $p$. We define:
\begin{defi} Let $S$ be a translation surface. The injectivity radius $r_p(S)$ of $S$ in $p$ is the supremum of all $r$, such that $U_r(p)$ is isometric to an open disk in $S$. We call the injectivity radius $r_{inj}$ of $S$ the infimum of all $r_p(S)$:
\[
    r_{inj}= \inf\{ r_p(S) \mid p \in S \}.
\]
\end{defi}
Now we prove the following theorem which is Theorem \ref{thm:SR_stratum} of the introduction:
\begin{thm}[Area estimate] Let $S$ be a translation surface in the stratum $\mathcal{H}(K)$, for $K=(k_1 ,\ldots, k_n)$, such that  
$k_1 \leq k_2 \leq \ldots\leq k_n$. Then
\[
    \sr(\mathcal{H}(K)) \leq \frac{4}{\pi \cdot (k_n+1)}.
\]      
\label{thm:SR_stratum2}
\end{thm}
To this end we first show the following lemma:
\begin{lem}  Let $S$ be a translation surface. Then $r_p(S) = \frac{1}{2} \ell(\mu_p)$, where $\mu_p$ is a shortest homotopically non-trivial geodesic loop with starting and endpoint $p$. Furthermore 
\[
    r_{inj}= \frac{\sy(S)}{2}.
\]
\label{thm:inj_systole}
\end{lem}
\begin{proof}  We first prove that $r_p(S) = \frac{1}{2} \ell(\mu_p)$. As $\mu_p$ is a homotopically non-trivial geodesic loop, we have that 
\[
\ell(\mu_p) \geq 2r_p(S).
\]
Set $R=r_p(S)$. To prove the other direction, we lift $U_R(p) = B_R(p)$ to the universal covering space $\tilde{S}$. Let $\left(B_R(p_i)\right)_{i \in \pi_1(S)}$ be the lifts of $B_R(p)$. Then the closure of two such disks $\overline{B_R(p_m)}$ and $\overline{B_R(p_l)}$  may intersect, but can only intersect at the boundary. Let without loss of generality $\overline{B_R(p_1)}$ and $\overline{B_R(p_2)}$ two such disks and let $q$ be an intersection point of $\overline{B_R(p_1)}$ and $\overline{B_R(p_2)}$. Let $\gamma_{p_1,q}$ and $\gamma_{p_2,q}$ be the geodesic arcs in  $\overline{B_R(p_1)}$ and $\overline{B_R(p_2)}$, respectively, that connect the respective centers and $q$.\\
We now show that $p_1 \neq p_2$, from which follows by covering theory that $pr( \gamma_{p_1,q} \cup \gamma_{p_2,q})$ is a non-trivial loop in $S$.\\
Suppose $p_1=p_2$. Then $\gamma_{p_1,q}$ and $\gamma_{p_2,q}$ are different geodesic arcs connecting $p_1$ and $q$. But by the Cartan-Hadamard Theorem there can be only one geodesic arc connecting two different points in a $\cat(0)$ space. A contradiction. Hence $pr( \gamma_{p_1,q} \cup \gamma_{p_2,q})$ is a loop $\mu'_p$ with base point $p$ of length $2r_p(S)$. Hence the shortest geodesic loop $\mu_p$ with base point $p$ has length  smaller than or equal to $2r_p(S)$. In total we have:
\[
      2r_p(S) = \ell(\mu_p).
\] 
By passing to the infimum we obtain  $r_{inj} = \frac{\sy(S)}{2}$, which is the second part of the statement in Lemma \ref{thm:inj_systole}. This concludes our proof. 
\end{proof}
Theorem \ref{thm:SR_stratum2} is a direct consequence of the following corollary:
\begin{cor} Let $S$ be a translation surface with a cone point of cone angle $2\pi k$. Then
\[
      \sr(S) \leq \frac{4}{\pi \cdot k}.
\]      
\label{thm:SR_cone_angle}
\end{cor}
\begin{proof}
Let $p$ be the cone point of cone angle $2\pi k$. Set $R=r_p(S)$. Then $U_R(p)=B_R(p)$ is an embedded disk of  radius $R$ in $S$. Hence
\[
R^2 \pi k  =  \area(B_R(p))  < \area(S) \text{ \ \  and \ \ } \frac{\sy(S)}{2}  =r_{inj} \leq r_p(S) = R.  
\]
Combining these two inequalities we obtain:
\[
      \sr(S) \leq \frac{4}{\pi \cdot k},
\]   
which proves Corollary \ref{thm:SR_cone_angle} and therefore Theorem \ref{thm:SR_stratum2}.
\end{proof}

We now prove that translation surfaces have short saddle connections by expanding embedded disks around cone points. The following theorem is Theorem \ref{thm:saddle} of the introduction: 
\begin{thm}[Short saddle connections]  Let $S$ be a translation surface of genus $g$ in the stratum $\mathcal{H}(K)$ where $K=(k_1 ,\ldots, k_n)$, such that $k_1 \leq k_2 \leq \ldots \leq k_n$. Then there exist $\lfloor \frac{n}{2} \rfloor$ saddle connections $(\delta_l)_{l=1 ,\ldots, \lfloor \frac{n}{2} \rfloor}$, such that 
\[
  \frac{\ell(\delta_1)^2}{\area(S)} \leq \frac{4}{\pi(2g-2+n)}  \text{ \ \ and \ \ } \frac{\ell(\delta_l)^2}{\area(S)}  \leq \frac{4}{\pi(2g+n - 2l- \left( \sum_{i=0}^{2l-3}   k_{n-i}  \right) )}, \text{ \ for \ } l \geq 2.
\]      
\label{thm:saddle2}
\end{thm}  
We also give a refined estimate in the stratum $\mathcal{H}(1,1,\ldots,1)$. 

\begin{cor}[Short saddle connections in  $\mathcal{H}(1,1,\ldots,1)$ ]Let $S$ be a translation surface of genus $g$ in the stratum $\mathcal{H}(1,1,\ldots,1)$. Then there exist $ g-1 $ saddle connections, such that 
\[
 \frac{\ell(\delta_1)^2}{\area(S)} \leq \frac{1}{\pi(g-1)} \text{ \ \ and \ \ }  \frac{\ell(\delta_l)^2}{\area(S)}  \leq \frac{1-\frac{\pi\ell(\delta_1)^2 \cdot l}{\area(S)}}{\pi(g -l)}  \text{ \ for \ } l \geq 2. 
\]    
\label{thm:saddle2_111}
\end{cor} 
Simplifying the second inequality for $\frac{\ell(\delta_l)^2}{\area(S)}$  we obtain Corollary \ref{thm:saddle_111} of the introduction.

\begin{proof} Let $S$ be a translation surface in the stratum $\mathcal{H}(K)$, where $K=k_1,\ldots,k_n$ and
\[
     k_1 \leq k_2 \leq \ldots \leq k_n.
\]   
We recall that $S$ has $n$ cone points with respective cone angles 
\[
2\pi \cdot (k_i+1) \text{ \ with \ }   \sum_{i=1}^n k_i = 2g-2
\]
where $g$ is the genus of $S$. Let $(p_i)_{i=1 ,\ldots, n}$ be the cone points of $S$ and for a fixed $i$ let $B_\epsilon(p_i)$ be an embedded disk of radius $\epsilon > 0$ around $p_i$. Note that we do not assume that $2\pi (k_i+1)$ is the cone angle at $p_i$ as this would not conform with our procedure.
The idea is to now expand the radii of these disks successively in $ \lfloor \frac{n}{2} \rfloor$ steps until they intersect. This will give us in each step a saddle connection together with an upper bound based on the area of the respective disks.  We start with the first step as follows:\\ 
\\
\textit{Step 1:} We expand the radii of the $n$ disks $(B_\epsilon(p_i))_i$ simultaneously until either the closure of two disks with radius $r_1$ intersect or the closure of a single disk with radius $r_1$ self-intersects. In the first case, we assume without loss of generality that the two disks $\overline{B_{r_1}(p_1)}$ and  $\overline{B_{r_1}(p_2)}$ intersect. In the second case, we assume that $\overline{B_{r_1}(p_1)}$ self-intersects. Connecting the respective saddle points by a geodesic arc, we obtain a saddle connection $\delta_1$ of length $\ell(\delta_1) \leq 2r_1$.   
We would like to mention that there is also a 'degenerate case', where we have more than one self-intersecting disk or more than two disks intersecting at radius $r_1$. In this case we arbitrarily choose a single such disk or a pair of disks and treat the other disks successively in the following steps. In this case, the next radius  $r_2$ in our sequence will be equal to $r_1$.
Let $2\pi(m_i+1)$ be the cone angle at $p_i$. Then for the union of all disks of radius $r_1$ we obtain:
\[
\area(\biguplus_{i=1}^n B_{r_1}(p_i)) = \sum_{i=1}^{n} \area(B_{r_1}(p_i)) = \sum_{i=1}^{n}  \pi  (m_i + 1) \cdot r_1^2 = \pi (2g-2 +n) \cdot r_1^2.
\]
As the union of the disks $\biguplus_{i=1}^n B_{r_1}(p_i)$ is embedded in $S$, we have furthermore
\[
\pi  (n+2g-2) \cdot r_1^2 = \area(\biguplus_{i=1}^n B_{r_1}(p_i)) \leq \area(S). 
\]
As $\ell(\delta_1) \leq 2r_1$ or $\frac{\ell(\delta_1)^2}{4} \leq r_1^2$ we obtain from the above inequality an upper bound for $\frac{\ell(\delta_1)^2}{\area(S)}$:
\[
\frac{\ell(\delta_1)^2}{\area(S)} \leq \frac{4}{\pi \cdot (2g-2 +n)}.
\]
In particular the last inequality implies an upper bound for the quotient of the square of the length of the shortest saddle connection and the area which is only slightly weaker than the bound of Boissy and Geninska in  Theorem \ref{thm:saddle_nohorse}. \\
\\
\textit{Step 2:} We note that we have at least $n-2$ remaining disks. We now expand the remaining disks until one of the following situations occurs: 
\begin{itemize}
\item[i)] the closure of a single disk among these disks self-intersects at radius $r_2$. Let without loss of generality $B_{r_2}(p_3)$ be that disk, or
\item[ii)] two different disks, both with radius $r_2$, or one with radius $r_1$ and the second with radius $r_2$ intersect. Here we assume that  $\overline{B_{r_2}(p_3)}$ and  $\overline{B_{r_2}(p_4)}$ intersect in the first case or $\overline{B_{r_2}(p_3)}$ and  $\overline{B_{r_1}(p_1)}$ intersect in the second case. 
\item[iii)] several instances of intersecting disks as described in the previous two cases occur. Again, in this 'degenerate case' we choose arbitrarily either a disk from \textit{Case i)} or two disks from \textit{Case ii)} for this step and treat the other disks successively in the next step or steps. In our result this would just imply that $r_3 = r_2$. 
\end{itemize}
In both \textit{Case i)} and \textit{Case ii)}, we obtain a second saddle connection $\delta_2$ by connecting the respective saddle point or saddle points with a geodesic arc $\delta_2 \neq \delta_1$. In the first case $\ell(\delta_2) \leq 2r_2$ and in the second case $\ell(\delta_2) \leq r_1 + r_2$ or $\ell(\delta_2) \leq 2r_2$. So in any case $\ell(\delta_2) \leq 2r_2$. As in \textit{Step 1}, we obtain an upper bound on $r_2$, as all disks are embedded. In any case we have
\begin{eqnarray}
     \sum_{i=3}^{n} \area(B_{r_2}(p_i))       \leq    \area(B_{r_1}(p_1)) + \area(B_{r_1}(p_2)) + \sum_{i=3}^{n} \area(B_{r_2}(p_i)) \leq \area(S).
\label{eq:disk_all}     
\end{eqnarray}  
We recall that $k_1 \leq k_2 \leq \ldots \leq k_n$ and that $2\pi(m_i+1)$ is the cone angle at $p_i$. By the formula for the area of a disk of radius $r_2$ we therefore obtain   
\[
   \pi  ( 2g -4 + n - k_n - k_{n-1} ) r^2_2 \leq \pi  ( 2g -4 + n - m_1 - m_2 ) r^2_2  \leq  \sum_{i=3}^{n} \area(B_{r_2}(p_i)) \leq \area(S). 
\]
As $\ell(\delta_2) \leq 2 r_2$ and therefore $\frac{\ell(\delta_2)^2}{4} \leq r_2^2$ this implies
\begin{eqnarray}
 \nonumber
\frac{\ell(\delta_2)^2}{4} &\leq& r_2^2 \leq \frac{ \area(S)}{\pi(2g -4 +n - k_n - k_{n-1} )}  \text{ \ hence \ } \\
\frac{\ell(\delta_2)^2}{\area(S)}  &\leq& \frac{ 4}{\pi(2g -4 +n - k_n - k_{n-1} )}. 
\end{eqnarray}
We proceed this way by expanding in each step $l$ the remaining disks further in each step picking either a single self-intersecting disk or two different disks that intersect. In each step we obtain a new saddle connection together with an upper bound of its length. In the $l$-th step we have at least $n-2(l-1)$ remaining disks:\\
\\
\textit{Step l:}  We obtain a saddle connection $\delta_l$ such that $\ell(\delta_l) \leq 2r_l$ by connecting the respective saddle point or saddle points with a geodesic arc $\delta_l$ length smaller or equal to $2r_l$. As in \textit{Step 2}, we obtain an upper bound on $ r_l$
\begin{eqnarray}
\nonumber
\frac{\ell(\delta_l)^2}{4} &\leq& r_l^2 \leq \frac{ \area(S)}{\pi(2g+n - 2l- \left( \sum_{i=0}^{2l-3}   k_{n-i}  \right) ) }  \text{ \ hence \ } \\
\frac{\ell(\delta_l)^2}{\area(S)}  &\leq& \frac{4}{\pi(2g+n - 2l- \left( \sum_{i=0}^{2l-3}   k_{n-i}  \right) )}.
\label{eq:delta_l}
\end{eqnarray}
In the case of the largest stratum  $\mathcal{H}(1,1,\ldots,1)$ we have that $k_i =1$ for all $i$ and therefore $n=2g-2$. In this case we can get a more refined inequality by taking into account the area of all expanded disks.  As $r_1 \leq r_k$ for all $k$ we obtain in \textit{Step l}
\begin{eqnarray}
\nonumber
 \sum_{i=1}^{2l} \area(B_{r_1}(p_i)) &+& \sum_{i=2l+1}^{n} \area(B_{r_l}(p_i))  \leq \area(S)
     \text{ \ or \ } \\    \nonumber 
   4\pi l \cdot r_1^2  &+& \pi (4g-4l) \cdot r_l^2  \leq \area(S)  \text{ \ hence \ } \\
\frac{\ell(\delta_l)^2}{\area(S)}  &\leq& \frac{1-\frac{\pi\ell(\delta_1)^2 \cdot l}{\area(S)}}{\pi(g -l)}.   
\label{eq:delta_l111}
\end{eqnarray}
The algorithm ends after  $\lfloor \frac{n}{2} \rfloor$ steps if in each step we obtain a saddle connection between two new saddle points. Hence after $l = \lfloor \frac{n}{2} \rfloor$ steps we obtain Theorem \ref{thm:saddle2} using Equation \eqref{eq:delta_l}. We furthermore obtain Corollary \ref{thm:saddle2_111} using Equation \eqref{eq:delta_l111}. 
\end{proof}

Lemma \ref{thm:inj_systole} implies that if the systole of a translation surface is large, then it is embedded in a large disk. In this case the systole is also embedded in a large tube, as we will see in the following.\\ 
Let $\eta \subset S$ be a simple closed geodesic in $S$. We define a neighborhood $U_w(\eta)$ around $\eta$ of width $w > 0$ by
\[
U_w(\eta)=\left\{p \in S \mid \dist(p,\eta) < w \right\}.
\]
For sufficiently small $w > 0$ the region $U_w(\eta)$ is a topological annulus. We increase $w$ until for some $w = \omega_{\eta}$ the closure $\overline{U_{\omega_{\eta}}(\eta)}$ of this region will start to \emph{self-intersect}, i.e. there exist two geodesic arcs $\delta'$ and $\delta''$ of length $\omega_{\eta}$ emanating from $\eta$ and having the endpoint $p$ in common. We call this value $\omega_{\eta}$ the \textit{maximal collar width} of $\eta$. As $U_w(\eta)$ is open, $\omega_{\eta}$ is the maximum of all $w$, such that $U_w(\eta)$ is a topological annulus in $S$. Finally, for $w < \omega_{\eta}$ we call 
\[
C_w(\eta)= U_w(\eta) = \left\{p \in S \mid \dist(p,\eta) < w  \right\}
\]
a \textit{collar} or \textit{cylinder} around $\eta$ of width $w$. We have:

\begin{lem}[Collar lemma for systoles] Let $\alpha$ be a systole of a translation surface $S$. Then for the the maximal collar width $\omega_\alpha$ of $\alpha$ we have
$$      \omega_{\alpha} \geq \frac{\ell(\alpha)}{4}. $$
\label{thm:cyl_sys}
\end{lem}

The lemma uses similar arguments as the proof for the disks. A version for hyperbolic Riemann surfaces, which uses the same arguments can be found in \cite{am1}. For the sake of completeness we repeat the proof here.
\begin{proof} Let $\alpha$ be a systole of a translation surface $S$ of genus $g \geq 2$. The closure $\overline{U_{\omega_{\alpha}}(\alpha)}$ of the annulus of width $\omega_\alpha$ self-intersects in a point $p$. This means that there exist two geodesic arcs $\delta'$ and $\delta''$ of length $\omega_{\alpha}$ emanating from $\alpha$ and having the endpoint $p$ in common. These two arcs meet $\alpha$ at an angle $\theta \geq \frac{\pi}{2}$ and form a geodesic arc $\delta$. 
  This arc $\delta$ can have a single endpoint on $\alpha$ or two different ones. In the first case $\delta$ is a simple closed curve which must have at least the length of a systole. Therefore
\[
       \ell(\delta) = 2 \omega_\alpha \geq \ell(\alpha).         
\] 
This implies that $\omega_\alpha \geq \frac{\ell(\alpha)}{2}$, especially $\omega_\alpha \geq \frac{\ell(\alpha)}{4}$ and our statement is true.
\\ 
In the second case $\delta$ has two different endpoints on $\alpha$. In this case the endpoints of $\delta$ on $\alpha$ divide $\alpha$ into two parts. We denote these two arcs on $\alpha$ by $\alpha'$ and $\alpha''$. Let without loss of generality $\alpha'$ be the shorter arc of these two. So we have that
\[
   \ell( \alpha') \leq \frac{\ell(\alpha)}{2}.
\]
We note that $\delta$ is not freely homotopic with fixed endpoints to $\alpha'$ or $\alpha''$ as the universal covering of $S$ is a global $\cat(0)$ space. 
 Let $\beta$ be the simple closed geodesic in the free homotopy class of $\alpha'\cdot \delta$, where $\cdot$ denotes the concatenation of the two paths. As $\alpha$ is a systole of $S$, we have that
\begin{equation}
 \ell(\alpha) \leq   \ell(\beta) \leq \ell(\alpha') + \ell(\delta) \leq \frac{\ell(\alpha)}{2} + 2 \omega_{\alpha}, \text{ \ \ hence \ } \frac{\ell(\alpha)}{4} \leq  \omega_{\alpha}.
\label{eq:omega_alpha}   
\end{equation}
This proves our statement in the second case and concludes the proof.   
\end{proof}
Note that equality in Equation \eqref{eq:omega_alpha} can indeed occur if the endpoints of $\delta$ on $\alpha$ divide $\alpha$ into two parts of equal length and are both cone points of the surface $S$. Each systole $\alpha$ in a translation surface $S$ has an embedded collar $C_{w}(\alpha)$ of width $w \geq \frac{\ell(\alpha)}{4}$. From this fact we also obtain an upper bound for the length of a systole via an area argument. However, this estimate is not better than the one given in the introduction in inequality \eqref{eq:srh_up}. Next we prove that translation surfaces have the following property: 
\begin{prop}\label{prop:ConePoint} Let $S$ be a translation surface, then there exists a systole of $S$ that passes through a cone point.
\label{thm:sys_tr}
\end{prop}

\begin{proof}[Proof] Let $S$ be a translation surface and let $\gamma$ be a simple closed geodesic in $S$ that does not intersect a cone point. Let $\epsilon > 0$ be a sufficiently small positive real number such that 
\[
C_\epsilon(\gamma)=\left\{p \in M \mid \dist(p,\gamma) < \epsilon \right\}
\]
is a flat cylinder around $\gamma$ that does not contain any cone points. Expand the cylinder until at width $w$ $\overline{C_w(\gamma)}$ intersects a cone point $p$ at its boundary 
\[
\partial \overline{C_w(\gamma)} = \partial_1 \overline{C_w(\gamma)} \cup \partial_2 \overline{C_w(\gamma)},
\]
consisting of the two connected components $\partial_1 \overline{C_w(\gamma)}$ and $\partial_2 \overline{C_w(\gamma)}$. Let without loss of generality $\gamma' =  \partial_1 \overline{C_w(\gamma)}$ be the  boundary part containing the cone point $p$. Now $p$  might divide $\gamma'$ into two or more simple closed geodesics, or not. In the first case, let $\gamma''$ be such a geodesic, that is contained in $\gamma'$ and that contains $p$; in the second case, set $\gamma''=\gamma'$. Now 
\[
    \ell(\gamma)=  \ell(\gamma') \geq \ell(\gamma'').
\]
Hence, for each simple closed geodesic $\gamma$ there exists a simple closed geodesic $\gamma''$ of equal or smaller length than $\gamma$ that passes through a cone point (see also \cite{ma}, Lemma 4.1.2). Hence the minimum $\sy(S)$ is attained in at least one simple closed geodesic that passes through a cone point, from which follows Proposition \ref{thm:sys_tr}. 
\end{proof}

Finally we prove Theorem \ref{thm:sr_cyclic} of the introduction:

\begin{thm}[Intersystolic inequalities] \label{Ii} Let $ \sr^{tr}(g)$ and  $\sr_h^{tr}(g)$ be the supremal systolic ratio and homological systolic ratio in genus $g$. Then 
\[
                \frac{\sr^{tr}(g)}{k}  \leq   \sr^{tr}(k(g-1)+1)   \text{ \ and \ }   \frac{\sr_h^{tr}(g)}{k}  \leq   \sr_h^{tr}(k(g-1)+1).
\]
\label{thm:sr_cyclic2}
\end{thm}
\begin{proof} Let $S$ be a surface of genus $g$. Recall that
  every translation surface contains infinitely many regular closed geodesics, i.e.
  closed geodesics which do not contain a cone point. This
  was shown for $g \geq 2$ in \cite{Masur}, Theorem 2 and can be directly seen for 
  flat tori. 
  Regular closed geodesics are non-separating, since  by the Poincar{\'e} recurrence theorem every  trajectory leaving the geodesic in a fixed transverse direction $v$ returns to the geodesic
  or hits a singularity. There are only finitely many trajectories which hit
  a singularity before coming back to the geodesic. Every returning trajectory
  connects the two sides of the geodesics in its complement. \\
  We now cut $S$ along a regular closed geodesic to obtain a connected surface $S^c$ with two boundary geodesics $\alpha_1$ and $\alpha_2$. 
We then construct a cyclic cover $\tilde{S}$ of $S$ by pasting $k$ copies $(S^c_i)_{i=1,\ldots,k}$ of $S^c$ with boundary curves $\alpha^i_1$ and $\alpha^i_2$ together.  
To this end we identify the boundaries of the different $\left(S_i^c\right)_{i=1,..,k}$ in the following way
\begin{equation}
    \alpha^{k}_1 \sim \alpha^1_2  { \ \ and \ \ } \alpha^{i}_1 \sim \alpha^{i+1}_2  \text{ \ for \ } i=1,...,k-1
\label{eq:paste1}
\end{equation}
to obtain a cyclic cover. We denote the surface of genus $k(g-1)+1$ obtained according to this pasting scheme as
\[
    \tilde{S}= S_1^c + S_2^c + ... + S_k^c~\mod (\ref{eq:paste1}).
\]
As the covering is cyclic we have for the systole and homological systole of $S$:
\[
      \sy(S) \leq \sy(\tilde{S})   \text{ \ and \ }   \syh(S) \leq \syh(\tilde{S}).
\]
Theorem \ref{thm:sr_cyclic2} then follows by taking a maximal surface in the case of $\sr^{tr}(g)$ or a sequence of surfaces $(S_n)_n$ whose systole length converges to $\sr_h^{tr}(g)$ in the case of $\sr_h^{tr}(g)$.
\end{proof}

We note that in our construction we do not cut $S$ through a cone point. Therefore we obtain a covering surface $\tilde{S}$ with a controlled number of cone points. This means if $S$ is in the stratum $\mathcal{H}(K)$, where $K=(k_1 ,\ldots, k_n)$ and $\tilde{S}$ is a cyclic cover of order $l$ then
\[
    \tilde{S} \in \mathcal{H}(K^l) , \text{ \ where \ }  K^l=\underbrace{(K,K,\ldots,K)}_{l \,\,times}.
\]
Hence we also obtain:
\begin{cor} Let $\sr(\mathcal{H}(K))$ and $\srh(\mathcal{H}(K))$ and supremal systolic ratio and homological systolic ratio in the stratum $\mathcal{H}(K)$, where $K=(k_1 ,\ldots, k_n)$ and let $K^l=\underbrace{(K,K,\ldots,K)}_{l \,\,times}$.   Then 
\[
        \frac{\sr(\mathcal{H}(K))}{l}  \leq  \sr(\mathcal{H}(K^l))   \text{ \ and \ }    \frac{\srh(\mathcal{H}(K))}{l}  \leq  \srh(\mathcal{H}(K^l)).
\]
\label{thm:sr_cyclic_strata}
\end{cor}

\section{A property of maximal surfaces}
In this section we prove that in a maximal surface $S_{max}$ of a given stratum every simple closed geodesic that does not pass through a cone point is intersected by a systole. This will prove the following theorem which is Theorem \ref{thm:char_Smax} of the introduction:
\begin{thm} Suppose that $S_{max}$ is  a maximal translation surface in the stratum $\mathcal{H}(K)$, for $K=(k_1 ,\ldots, k_n)$. Then every simple closed geodesic, that does not run through a cone point is intersected by a systole of $S_{max}$. 
\label{thm:char_Smax2}
\end{thm}
Observe that we explicitly use the existence of a maximal surface in this proof. \cite{jp} explicitly construct maximal
surfaces in $\mathcal{H}(2g-2)$ and conjecture a maximal surface in $\mathcal{H}(1,1)$.
However, it is to our knowledge not known in general, whether each stratum contains a maximal surface. Whereas it is true, that
the full moduli space $\mathcal{M}_g^{tr}$ does contain a surface with a systole of maximal length for the following reason.
For fixed genus $g$ consider a sequence $(S_n)_{n \geq 1}$ of surfaces whose systolic ratio converges to the supremum, i.e. 
\[
       \lim_{n \to \infty} \sr(S_n) = \sr^{tr}(g).
\]
As $\sr^{tr}(g) \geq \frac{0.58}{g-1}$ we know that this sequence does not converge to the boundary of the moduli space $\mathcal{M}_g^{tr}$. Hence the maximum is attained, as 
\[
            M_g = \{ S \in \mathcal{M}_g^{tr} \mid \sr(S) \geq \frac{0.58}{g-1} \} 
\]
is compact. Note that for the homological systole we do not know if the systole length converges to zero if $\syh(\cdot)$ of a sequence of surfaces converges to $\sr_h^{tr}(g)$. Therefore the same argument might not apply in this case. 

\begin{proof}[Proof of Theorem \ref{thm:char_Smax2}]
  Let $S_{max}$ be a maximal translation surface in the stratum $\mathcal{H}(k_1,\ldots,k_n)$. Let $\gamma$ be a simple closed geodesic that does not pass through a cone point. Now, for some $\epsilon > 0$, $\gamma$ is embedded in a flat cylinder $C_\epsilon(\gamma)$ that does not contain a cone point. Assume that no systole intersects $\gamma$. As every geodesic, that intersects $C_\epsilon(\gamma)$ intersects $\gamma$, we have that no systole intersects $C_\epsilon(\gamma)$.
Now, the length spectrum of a translation surface is discrete. Hence for any simple closed geodesic $\eta$ that is not a systole, we have:
\[
    \ell(\eta) > \sy(S_{max}) + \delta, \mbox{ for some } \delta > 0 \mbox{ independent of $\eta$.}
\]
    
We can construct a new translation surface $S'_{max}$ from $S_{max}$ by replacing the cylinder $C_\epsilon(\gamma)$ by a smaller cylinder $C_{\epsilon'}(\gamma)$ of width $\epsilon > \epsilon' > 0$.
We show below that we can choose $\epsilon'$ in a way such that no simple closed geodesic $\eta'$ not homotopic to $\gamma$ and intersecting $C_{\epsilon'}(\gamma)$ in $S'_{max}$ is shorter than $\sy(S_{max})$. This way we can construct a comparison surface $S'_{max}$, such that 
\[
         \sy(S'_{max}) = \sy(S_{max}) \text{ \ but \ }  \area(S'_{max}) < \area(S_{max}), \text{ \ hence \ }  \sr(S'_{max}) > \sr(S_{max}).
\]
But this is a contradiction to the fact that $S_{max}$ is maximal. Hence our assumption that $\gamma$ is not intersected by a systole is wrong. Therefore any simple closed geodesic that does not contain a cone point is intersected by a systole.\\
It remains to show the existence of an $\epsilon'$ as stated above. We denote in the following $ s_{\fs max} = \sy(S_{max})$. We will define $\epsilon'$ as $\epsilon' = r\cdot \epsilon$ with $0 < r < 1$ large enough. Let  $\eta'$ be a geodesic closed curve on $S'_{max}$ not homotopic to $\gamma$ which intersects $\gamma$  with intersection number $N$. For the moment we only suppose for $r$ that  $r > \frac{1}{2}$.\\
If $N$ is large enough, more precisely if $N > \frac{2\smax}{\varepsilon}$, then we have that \[\ell(\eta') > N \cdot \epsilon' > N\cdot\frac{\epsilon}{2} > \smax \hspace*{10mm}\mbox{(cf. Figure~\ref{fig:cuts1})}.\]
Hence $\eta'$ has the desired property.\\ 
\begin{figure}[h!]
    \centering
        \begin{tikzpicture}[scale=1,>=stealth',shorten >=1pt,auto,node distance=3cm,thick]
          \tikzstyle{place}=[circle,thick,draw=black!,minimum size=6mm]
          \tikzstyle{bullet}=[circle,thick,fill=black!,minimum size=2mm]
          \draw (-5,1) -- (5,1);
          \draw (-5,0) -- (5,0);
          \draw (-5,-1) -- (5,-1);
          \draw (-4.5,-1.5) -- (-3.5,1.5);  \draw (-1.5,-1.5) -- (-.5,1.5);  \draw (3,-1.5) -- (4,1.5);
          \draw (3.82,-1) -- (3.82,1.02);
          \node (x1) at (-4,1.3) {$\eta'$};
          \node (x5) at (-1,1.3) {$\eta'$};
          \node (x6) at (4.3,1.3) {$\eta'$};
          \node (x2) at (5.2,0) {$\gamma$};
          \node (x3) at (6.5,0) {$C_{\epsilon'}(\gamma)$};
          \node (x4) at (4.1,.4) {$\varepsilon'$};
          \draw[thick,decorate,decoration={brace,amplitude=8pt,raise=5pt}] (3,-1) -- node [above, left = 10pt] {$> \varepsilon'$} (3.8,.6);
         \end{tikzpicture} 
        \caption{The curve $\gamma$ is intersected $N$ times by a geodesic $\eta'$ with large $N$}
        \label{fig:cuts1}
\end{figure}

Let us now consider the case that $N \leq \frac{2\smax}{\varepsilon}$.
For this we will have to use a bigger lower bound for $r$. The curve $\eta'$ can be decomposed as \[\eta' = \eta'_0 \cup \eta'_1 \cup \ldots \cup \eta'_N,\] where $ \eta'_1$, \ldots, $\eta'_N$ are the connected components of $\eta' \cap C_{\epsilon'}(\gamma)$ and $\eta'_0$ is ``the rest'', i.e. $\eta'_0$ is the relative closure  of $\eta' \cap (S'_{max}\backslash  C_{\epsilon'})$ (cf. Figure~\ref{fig:cuts2}). In particular we have:
\[\ell(\eta') = \ell(\eta'_0) +  \ell(\eta'_1) + \ldots + \ell(\eta'_n)\]
Furthermore, each $\eta'_i$ is a geodesic segment with a fixed direction. It is the hypotenuse in an Euclidean right triangle with a vertical cathetus of length $\epsilon'$ and a horizontal cathetus whose length we call $d_i$. In particular, we have that $\ell(\eta'_i)^2 = \epsilon'^2 + d_i^2$  for  $i \in \{1,\ldots, n\}$ (cf. Figure~\ref{fig:cuts3}).
\begin{figure}[h!]
  \centering
        \begin{tikzpicture}[scale=1,>=stealth',shorten >=1pt,auto,node distance=3cm,thick]
          \tikzstyle{place}=[circle,thick,draw=black!,minimum size=6mm]
          \tikzstyle{bullet}=[circle,thick,fill=black!,minimum size=2mm]
          \draw (-5,1) -- (5,1);
          \draw (-5,0) -- (5,0);
          \draw (-5,-1) -- (5,-1);
          \draw (-5,-2) -- (-3,2);  \draw (-2,-2) -- (0,2);  \draw (3,-2) -- (5,2);
          \node (x1) at (-3.6,1.6) {$\eta_0'$};
          \node (x5) at (-.6,1.6) {$\eta_0'$};
          \node (x6) at (4.4,1.6) {$\eta_0'$};
          \node (x7) at (-5.17,-1.6) {$\eta_0'$};
          \node (x8) at (-2.15,-1.6) {$\eta_0'$};
          \node (x9) at (2.9,-1.6) {$\eta_0'$};
          \node (x8) at (-4.5,-.35) {$\eta_1'$};
          \node (x9) at (-1.5,-.35) {$\eta_2'$};
          \node (x10) at (3.5,-.35) {$\eta_3'$};              
          \node (x2) at (5.2,0) {$\gamma$};
          \node (x3) at (6.5,0) {$C_{\epsilon'}(\gamma)$};
         \end{tikzpicture} 
        \caption{Decomposition $\eta' = \eta_0' \cup \eta_1' \cup \ldots \cup \eta_N'$}
        \label{fig:cuts2}
      \end{figure}

Now we consider the corresponding non-geodesic curve $\hat{\eta}$ on  $S_{max}$ obtained from $\eta'$ as follows: Outside the cylinder  $C_\epsilon(\gamma)$, respectively the cylinder  $C_\epsilon'(\gamma)$, they coincide. From the construction of the surface we have that the boundary of the cylinder $C_\epsilon(\gamma)$ on $S_{max}$ can be identified with the boundary of  $C_\epsilon'(\gamma)$ on $S'_{max}$. For each segment $\eta'_i$ of $\eta'$ ($i \in \{1,\ldots, n\}$) we consider its end points $p_i$ and $q_i$ and can interpret them as well as boundary points of $C_\epsilon(\gamma)$. We define $\hat{\eta}_i$ to be the unique geodesic segment in the cylinder  $C_\epsilon(\gamma)$ connecting $p_i$ and $q_i$. Then we have that $ \hat{\eta} = \hat{\eta}_0 \cup \hat{\eta}_1 \cup \ldots \cup \hat{\eta}_N$ and
\[\ell(\hat{\eta}) = \ell(\hat{\eta}_0) +  \ell(\hat{\eta}_1) + \ldots + \ell(\hat{\eta}_n),\quad \ell(\hat{\eta}_0) = \ell(\eta'_0) \mbox{ and } \ell(\eta'_i) < \ell(\hat{\eta}_i) \mbox{ for } i \in \{1,\ldots, N\}. \]
Similarly as above $\hat{\eta}_i$ is the hypotenuse in an Euclidean right triangle with a vertical cathetus of length $\epsilon$ and a horizontal cathetus of the same length $d_i$ from above. In particular we have that $\ell(\hat{\eta}_i)^2 = \epsilon^2 + d_i^2$ (cf. Figure~\ref{fig:cuts3}).

\begin{figure}[h!]
  \centering
  \begin{tikzpicture}[scale=1,>=stealth',shorten >=1pt,auto,node distance=3cm,thick]
    \tikzstyle{place}=[circle,thick,draw=black!,minimum size=6mm]
    \tikzstyle{bullet}=[circle,thick,fill=black!,minimum size=2mm]
    \draw (-3,1) -- (3,1);
    \draw (-3,-1) -- (3,-1);
    \draw (-2,-2) -- (0,2);  
    \draw (-.5,1) -- (-.5,-1);
    \node (x1) at (-.3,0) {$\varepsilon'$};
    \node (x5) at (-.9,-1.3) {$d_i$};
    \node (x9) at (-1.4,-.15) {$\eta_i'$};
    \node (x3) at (2.5,0) {$C_{\epsilon'}(\gamma)$};
    \vertex[bullet, label = {[shift={(0.0,0.06)}]$p_i$}] (pi)  at (-.5,1) {};
    \vertex[bullet, label = {[shift={(0,-.8)}]$q_i$}] (qi)  at (-1.5,-1) {};
    \begin{scope}[shift={(-1,0)}]
    \draw (7,2) -- (13,2);
    \draw (7,-2) -- (13,-2);
    \draw (8.5,-2) -- (9.5,2);
    \draw (9.5,2) -- (11,3.5);
    \draw (8.5,-2) -- (7.5,-3);
    \draw (9.5,2) -- (9.5,-2);
    \node (x1) at (9.8,0) {$\varepsilon$};
    \node (x5) at (9.1,-2.5) {$d_i$};
    \node (x9) at (8.5,-.35) {$\hat{\eta_i}$};
    \node (x3) at (12.5,0) {$C_{\epsilon}(\gamma)$};
    \vertex[bullet, label = {[shift={(0.0,0.06)}]$p_i$}] (pi)  at (9.5,2) {};
    \vertex[bullet, label = {[shift={(0,-.8)}]$q_i$}] (qi)  at (8.5,-2) {};
    \end{scope}
  \end{tikzpicture} 
  \caption{Left side: Right triangle with hypotenuse $\eta_i'$ \quad Right side: Right triangle with hypotenuse $\hat{\eta}_i$}
  \label{fig:cuts3}
\end{figure}

Let $\eta$ be a geodesic homotopic to $\hat{\eta}$. Then $\eta$ intersects $\gamma$ as well with intersection number $N$ and thus by our assumption is not a systole. Therefore we have
\[\ell(\hat{\eta}) > \ell(\eta) > \smax + \delta.\]
We now obtain the following computation for $\ell(\hat{\eta_i})^2 - \ell(\eta'_i)^2$:
\[
\ell(\hat{\eta}_i)^2 - \ell(\eta'_i)^2 =  \epsilon^2 + d_i^2 - (\epsilon'^2 + d_i^2) = \epsilon^2 - \epsilon'^2 = (1-r^2)\cdot \epsilon^2.
\]
Recall that $\ell(\hat{\eta}_i) > \ell(\eta'_i) > 0$.  This leads to the following estimation of $\ell(\hat{\eta_i}) - \ell(\eta_i')$:
\[(\ell(\hat{\eta_i}) - \ell(\eta'_i))^2 <  \ell(\hat{\eta}_i)^2 - \ell(\eta'_i)^2 = (1-r^2)\cdot \epsilon^2\]
If we now choose $r$ such that we have in addition to $r > \frac{1}{2}$ that $\sqrt{1-r^2} < \frac{\delta}{2\smax}$ then we obtain:
\[\ell(\hat{\eta_i}) - \ell(\eta'_i) <  \sqrt{1-r^2}\cdot \epsilon <  \delta\cdot \frac{\epsilon}{2\smax} \]
This gives us for the length $\ell(\eta')$:
\[\begin{array}{lcl}
  \ell(\eta') &=& \ell(\eta'_0) + \sum_{i=1}^N\ell(\eta'_i) = \ell(\hat{\eta}_0) + \sum_{i=1}^N\ell(\eta'_i) > \ell(\hat{\eta}_0) +  \sum_{i=1}^N(\ell(\hat{\eta}_i) - \delta\cdot \frac{\epsilon}{2\smax})\\
  &=& \ell(\hat{\eta}) - N\cdot \delta\cdot \frac{\epsilon}{2\smax} > \ell(\hat{\eta}) - \delta \geq \ell(\eta) - \delta > \smax + \delta - \delta = \smax.
  \end{array}\]
This finishes the proof of  Theorem \ref{thm:char_Smax2}.

\end{proof}

\section{Systoles and the graph of saddle connections}\label{section:gos}

Suppose that $S$ is a translation surface of genus $g \geq 2$. Recall that it follows from Proposition~\ref{prop:ConePoint} that there exists a systole of $S$ which is a concatenation of saddle connections. This motivates us to consider the graph of saddle connections introduced in Definition~\ref{def:GraphOfSaddleconnections}. In this section we study how the systole of this graph and the systole of the surface are related.

\begin{defi}\label{def:GraphOfSaddleconnections}
The \emph{graph $\Gamma$ of saddle connections} of a translation surface $S$ is the following graph: Its vertices are the singularities of $S$. It has an edge between two vertices for each saddle connection which connects the two corresponding singularities. $\Gamma$ becomes a weighted graph by assigning to each edge the length of the corresponding saddle connection. 
\end{defi}

We write for an edge-path $c$ in a graph $c = c_1\ldots c_n$, if $c$ is the concatenation of the edges $c_1$, \ldots, $c_n$. Observe that the edge-path $c = c_1\ldots c_n$ defines the  path $\gamma = \gamma_1\cdot \ldots \cdot \gamma_n$  on $S$, where $\gamma_i$ is the saddle connection corresponding to $c_i$. We consider $\gamma$ as a directed path on $S$ and call it  \textit{the realization of $c$}. Reversely, any concatenation $\gamma$ of saddle connections on $S$ defines an edge-path $c$ in $\Gamma$.
The graph $\Gamma$ becomes a metric space by assigning to each edge its weight as length. We denote the {\em length of an edge-path} $c$ with respect to this metric by $l(c)$, i.e. $l(c)$ is the sum of the weights of the edges of $c$. We further consider the {\em combinatorial length} of $c$ which is the number of edges that $c$ contains. Recall that an edge path $c = c_1\ldots c_n$ has {\em backtracking}, if there are two consecutive edges $c_i$ and $c_{i+1}$ which are inverse to each other. The edge-path $c$ is called {\em reduced}, if it does not have backtracking. We say that $c$ is \emph{trivial}, if it is of combinatorial length 0, i.e. the only trivial edge-path is the empty path. 
We call a non-trivial closed reduced edge-path $c$ of minimal length $l(c)$ in $\Gamma$ {\em a systole of $\Gamma$}.\\[3mm]
We can  now use the graph $\Gamma$ for the study of systoles of the surface $S$ in the following way:
By Proposition~\ref{prop:ConePoint} we have  a systole $\gamma$ of $S$ which is a concatenation of saddle connections. Such a systole $\gamma$ defines a non-trivial reduced closed edge-path $c$ in $\Gamma$. However, $c$ has not to be a systole of $\Gamma$, since there might be shorter closed edge-paths in $\Gamma$ whose realizations are not geodesic, see Example~\ref{ex:Exception}. This happens if the realization is null-homotopic. Remark~\ref{prop:nullhomotopicSystoles} and Theorem~\ref{n3thm} describe the connection between systoles of the surface $S$ and systoles of the graph $\Gamma$.\\

As a main ingredient we use the following description of geodesics on translation surfaces by Dankwart from \cite{Dankwart}. Observe that what is called a 'local geodesic'  in \cite{Dankwart} is what we call a 'geodesic'  in our article.

\begin{lem}[{\cite{Dankwart}, Lemma 2.4}] \label{LemmaDankwart}
  A path $c: [0,T] \to  S$ is a geodesic if and only if it is continuous and a sequence of straight line segments outside the set  $\Sigma$ of singularities. At a singularity $x = c(t)$ the consecutive line segments make a flat angle at least $\pi$ with respect to both boundary orientations. 
\end{lem}

Motivated by Lemma~\ref{LemmaDankwart} we consider for closed edge-paths $c = c_1\ldots c_n$ in $\Gamma$ at each vertex $v_i$ between two edges $c_{i-1}$ and $c_i$ ($i$ indexed in $\ZZ/n\ZZ$) the two angles $\alpha_i$ and $\beta_i$ between the corresponding saddle connections $\gamma_{i-1}$ and $\gamma_i$. If for all $i$ these angles are at least $\pi$, we say that {\em all angles of $c$ are greater or equal to $\pi$}. By Lemma~\ref{LemmaDankwart} the realization $\gamma$ of a closed edge path $c$ is geodesic if and only if all angles of $c$ are greater or equal to $\pi$.

\begin{defi}
  Let $\Gamma$ be the graph of saddle connections of a translation surface $S$. We define:
  \[\GIG = \{c |\; c \mbox{ is a closed reduced edge-path in $\Gamma$ such that all angles of $c$ are greater or equal to $\pi$}\}\]
\end{defi}

Now, we have all notation to state how we determine systoles of the surface $S$ from the systoles of the graph of saddle connections $\Gamma$.

\begin{rem}\label{prop:nullhomotopicSystoles}
  Let $c$ be be a closed reduced edge-path in $\Gamma$. Its realization $\gamma$ is a systole of the surface $S$ if and only if  $\gamma$ is geodesic and $c$ is of minimal length among all  closed reduced edge-paths in $\Gamma$ whose realization is geodesic.
  This is the case, if and only if $c$ is an element of minimal length in $\GIG$.
\end{rem}

\begin{proof}
  The claim follow from Proposition~\ref{prop:ConePoint}, Lemma~\ref{LemmaDankwart} and the general fact that  geodesics\footnote{Recall that we use the term geodesic for what is called locally geodesic in \cite{Dankwart}} in local CAT(0)-spaces are never null homotopic.
\end{proof}

Example \ref{ex:Exception} shows an example of a translation surface which has a systole of the graph of saddle connections which is not a  systole of  the translation surface.  
However, this disturbing case only rarely occurs, as it follows from the following theorem which is Theorem~\ref{n3thm} of the introduction:

\begin{thm}\label{n3thmN}
  Let $c$ be a systole in the graph $\Gamma$ and let $\gamma$ be its realization. If the combinatorial length of $c$ is not 3, then $\gamma$ is a systole of the surface. If the combinatorial length of $c$ is 3 and all angles of $c$ are greater or equal to $\pi$, then $\gamma$ is a systole of the surface.
\end{thm}

For the proof of Theorem~\ref{n3thmN} we need to first have a closer look on Lemma~\ref{LemmaDankwart} of Dankwart. In Corollary~\ref{ShorteningLemma} we reprove part of the statement of this lemma, since the method we use will give us the ingredients we need to show Theorem~\ref{n3thmN}.
We use the following notations:
For a directed path $\gamma$ on $S$ we denote  its inverse path by $\gamma^{-}$.
    For two points $a$ and $b$ on $\gamma$ we denote by $\gamma_{a,b}$ the shortest subpath
  of $\gamma$ from $a$ to $b$, i.e. $\gamma_{a,b}$ is a subpath of $\gamma$
  which starts at $a$ and ends as soon as $\gamma$ hits the first time $b$. We further
  denote by $l(\gamma_{a,b})$ the length of $\gamma_{a,b}$.  Finally, an {\em embedded triangle} $\Delta$ of $S$ is the image of a map $\phi:T \to S$ with the following properties: $T$ is a triangle of the Euclidean plane $\RR^2$ such that $\phi(T^{\mbox{\fs int}})$ does not contain any singularity. Furthermore, $\phi|_{T^{\mbox{\fs int}}}$ is an embedding such that for each chart $(U,\tau)$ of the translation surface $S$ with $U \cap \phi(T^{\mbox{\fs int}}) \neq \emptyset$, we have that $\tau \circ \phi|_{\phi^{-1}(U)}$ is a translation. Here $T^{\mbox{\fs int}}$ denotes the interior of $T$. The {\em vertices} and {\em edges} of $\Delta$ are the images of the vertices and edges of $T$. Observe that the edges of $\Delta$ are by construction geodesics on $S$.

\begin{lem}\label{The2Case}
Let $\gamma_1$ be a saddle connection from $v_1$ to $v_2$ and $\gamma_2 \neq {\gamma_1}^{-}$ a saddle connection from $v_2$ to $v_3$ such that one of the two angles at $v_2$ between $\gamma_1$ and $\gamma_2$ is smaller than $\pi$. Then there is a  concatenation of saddle connections $\delta$ from $v_1$ to $v_3$ that is homotopic to $\gamma_1\cdot \gamma_2$ and shorter than $\gamma_1\cdot \gamma_2$. We allow that $v_1 = v_3$. Furthermore, $\delta$ is non-trivial, i.e. it defines a non-trivial path in the graph of saddle connections.
\end{lem}

\begin{proof}
  Let $x_1 \neq v_2$ be the point on $\gamma_2$ of largest distance from $v_2$ such that there exists an embedded Euclidean triangle $\Delta_1$ on $S$ with vertices $v_1$, $v_2$ and $x_1$ and edges $\alpha^{\sm{(1)}} = \gamma_1$ and $\beta^{\sm{(1)}} = (\gamma_2)_{v_2,x_1}$. Let $\gamma^{\sm{(1)}}$ be the third edge of $\Delta_1$ from $v_1$ to $x_1$. By the triangle inequality we have that
  \begin{equation}\label{Delta1}
    l(\gamma^{\sm{(1)}}) < l(\alpha^{\sm{(1)}}) + l(\beta^{\sm{(1)}}).
  \end{equation}
  \begin{center}
      \begin{figure}[h!]
        \begin{tikzpicture}[scale=1,<->,>=stealth',shorten >=1pt,auto,node distance=3cm,thick]
          \tikzstyle{place}=[circle,thick,draw=black!,minimum size=6mm]
          \tikzstyle{bullet}=[circle,thick,fill=black!,minimum size=2mm]
          \vertex[bullet, label = left:$v_1$](v1) at (0,0) {};
          \vertex[bullet, label = left:$v_2$](v2) at (6,8) {};
          \vertex[bullet, label = right:$v_3$](v3) at (10,0) {};
          \node [label = {[shift={(0.0,-0.35)}]$x_1$}] (x1)  at (7.2,6.2) {};
          \node [label = {$\beta^{\sm(1)}$}] (t1)  at ($ (v2) !.7! (x1) $) {};
          \node (Delta1) at (5,5.5) {$\Delta_1$};
          \vertex [bullet, label = below:{$w_1$}] (w1)  at ($ (v1) !.3! (x1) $) {};
          \node [label = {[shift={(0,0)}]below:$\gamma^{\sm(1)}$}] (t2)  at ($ (w1) !.6! (x1) $) {};
          \path[-, shorten >=-1mm] (v1) edge  node {$\alpha^{\sm(1)} = \gamma_1$} (v2);
          \path[- ,  shorten >= -1mm] (v2)  edge    node {} (v3);
          \path[-,  shorten >= 1mm] (v1) edge node {} (x1);
         \end{tikzpicture} 
        \caption{Filling in a triangle in a shortest closed concatenation of saddle connections on $S$.}
        \label{fig:Delta1}
      \end{figure}
  \end{center}
  We first consider the case that $x_1 = v_3$. In this case $\delta = \gamma^{\sm{(1)}}$ is a saddle connection from $v_1$ to $v_3$ and does the desired job.\\[3mm]
  Suppose now  $x_1 \neq v_3$ and therefore $x_1$ lies on the interior of $\gamma_2$. Then the geodesic segment $\gamma^{\sm{(1)}}$ from $v_1$ to $x_1$ contains a singularity which we call $w_1$ (cf. \textit{Figure \ref{fig:Delta1}}). If there are more than one singularities on $\gamma^{\sm{(1)}}$ then we choose the one closest to $x_1$. We consider the concatenation $\delta_1 = {\gamma^{\sm{(1)}}}_{w_1,x_1}\cdot {(\gamma_2)}_{x_1,v_3}$ of two geodesic segments. Its angle at $x_1$ is smaller than $\pi$. We can again fill in a triangle $\Delta_2$. More precisely, we choose the point $x_2 \neq x_1$ on ${(\gamma_2)}_{x_1,v_3}$ at largest distance from $x_1$ such that there exists an embedded triangle $\Delta_2$ with vertices $w_1$, $x_1$ and $x_2$ and edges $\alpha^{\sm{(2)}} = {{\gamma^{\sm{(1)}}}}_{w_1,x_1}$ and $\beta^{\sm{(2)}} = {(\gamma_2)}_{x_1,x_2}$. We denote its third edge  from $w_1$ to $x_2$  by $\gamma^{\sm{(2)}}$. Again by the triangle inequality we have:
  \begin{equation} \label{Delta2}l(\gamma^{\sm{(2)}}) < l(\alpha^{\sm{(2)}}) + l(\beta^{\sm{(2)}}).\end{equation}
  We denote $\gamma'_1 = {{\gamma^{\sm{(1)}}}}_{v_1,w_1}$  and $\delta_2 = \gamma_1'\cdot \gamma^{\sm{(2)}} \cdot {(\gamma_2)}_{x_2,v_3}$ (path in thick lines in \textit{Figure~\ref{fig:Deltak}}) and obtain:
  \[l(\delta_2) = l(\gamma_1') + l(\gamma^{\sm{(2)}}) + l({(\gamma_2)}_{x_2,v_3}) \stackrel{(\ref{Delta2})}{<} l(\gamma^{\sm{(1)}}) + l(\beta^{\sm{(2)}}) + l({(\gamma_2)}_{x_2,v_3}) \stackrel{(\ref{Delta1})}{<} l(\gamma_1) + l(\gamma_2).\]
  If $x_2 = v_3$, then the path $\delta = \delta_2$ is a concatenation of saddle connections from $v_1$ to $v_3$ which is by construction homotopic to $\gamma_1 \cdot \gamma_2$ and shorter than $\gamma_1 \cdot \gamma_2$, hence it does the desired job.\\[3mm]
  If $x_2 \neq v_3$, then we iterate the process of gluing in triangles: Suppose that for $k \geq 3$ the vertex $x_{k-1}$ of triangle $\Delta_{k-1}$ fulfills $ x_{k-1} \neq v_3$. This implies that its edge $\gamma^{\sm{(k-1)}}$ contains in its interior a singularity $w_{k-1}$.
  \begin{center}
    \begin{figure}[h!]
      \begin{tikzpicture}[scale=1,<->,>=stealth',shorten >=1pt,auto,node distance=3cm,thick]
      \tikzstyle{place}=[circle,thick,draw=black!,minimum size=6mm]
      \tikzstyle{bullet}=[circle,thick,fill=black!,minimum size=2mm]
            \vertex[bullet, label = left:$v_1$](v1) at (0,0) {};
            \vertex[bullet, label = left:$v_2$](v2) at (6,8) {};
            \vertex[bullet, label = right:$v_3$](v3) at (10,0) {};
            \node [label = {[shift={(0.0,-0.35)}]$x_1$}] (x1)  at (7.2,6.2) {};
            \node [label = {$\beta^{\sm(1)}$}] (t1)  at ($ (v2) !.7! (x1) $) {};
            \node (Delta1) at (5,5.5) {$\Delta_1$};
            \node [label = {[shift={(0,0.1)}]below:$\gamma_1'$}] (t5)  at ($ (v1) !.4! (w1) $) {};  
            \vertex [bullet, label = below:{$w_1$}] (w1)  at ($ (v1) !.3! (x1) $) {};
            \node  [label  = {[shift={(0.2,-.25)}]$x_2$}] (x2)  at ($ (v2) !.5! (v3) $) {};
            \node (Delta2) at (6.3,4.3) {$\Delta_2$};
             \node [label = {[shift={(.5,0)}]$\gamma^{\sm(2)}$}] (t3)  at ($ (w1) !.55! (x2) $) {};  
            \node [label = {[shift={(.1,0)}]$\gamma^{\sm(1)}$}] (t2)  at ($ (w1) !.2! (x1) $) {};
            \node [label = {[shift = {(-.15,0)}]below:$\alpha^{\sm(2)}$}] (t3)  at ($ (w1) !.55! (x1) $) {};
            \node [label = {[shift = {(0.1,0.2)}]above:$\beta^{\sm(2)}$}] (t4)  at ($ (v2) !.8! (x2) $) {};
            \node  [label  = {[shift={(0.25,-.25)}]$x_3$}] (x3)  at ($ (x2) !.6! (v3) $) {};
            \vertex [bullet, label = below:{$w_2$}] (w2)  at ($ (w1) !.3! (x2) $) {};
            \node [label = {[shift={(0,0.1)}]below:$\gamma_2'$}] (t6)  at ($ (w1) !.55! (w2) $) {};
            \node (Delta3) at (7.2,2.8) {$\Delta_3$};
            \vertex [bullet, label = below:{$w_3$}] (w3)  at ($ (w2) !.5! (x3) $) {};
            \node [label = {[shift={(0,0.15)}]below:$\gamma_3'$}] (t7)  at ($ (w2) !.55! (w3) $) {};            
            \node  [label  = {[shift={(0.25,-.25)}]$x_4$}] (x4)  at ($ (x3) !.6! (v3) $) {};
            \node (Delta4) at (9,1.3) {$\Delta_4$};
            \node  (delta4) [label =  {[shift={(0.35,0.4)}]$\mathbf{\delta_2}$}] at ($ (x2) !.4! (v3) $) {};

            \vertex [bullet, label = below:{$w_4$}] (w4)  at ($ (w3) !.4! (x4) $) {};
            \node [label = {[shift={(0,0.1)}]below:$\gamma_4'$}] (t8)  at ($ (w3) !.45! (w4) $) {};            
            \node [label = {[shift={(0,0.1)}]below:}] (t9)  at ($ (w4) !.45! (v3) $) {};    
            \path[-, shorten >= -2mm, line width = 3 ] (w1) edge node{} (x2);
            \path[-, shorten >= -1mm, line width = 3 ] (v1) edge node{} (w1);
            \path[-, shorten <= -3mm, shorten <= -1mm ,shorten >= -2mm, line width = 3 ] (x2) edge node{} (v3);
            \path[-, shorten >=-1mm] (v1) edge  node {$\alpha^{\sm(1)} = \gamma_1$} (v2);
            \path[-] (v2)  edge    node {} (v3);
            \path[-,  shorten >= 1mm] (v1) edge node {} (x1);

            \path[- , shorten >= -2mm] (w1) edge node {} (x2);

           \path[- , shorten >= -2mm] (w2) edge node {} (x3);
            \path[- , shorten >= -2mm] (w3) edge node {} (x4);
            \path[- , shorten >= -2mm] (w4) edge node {} (v3);
 
          \end{tikzpicture} 
    \caption{Iteration of filling in triangles in a shortest closed concatenation of saddle connections on $S$.}
    \label{fig:Deltak}
    \end{figure}
  \end{center}

  The concatenation ${\gamma^{\sm{(k-1)}}}_{w_{k-1},x_{k-1}}\cdot (\gamma_2)_{x_{k-1},v_3}$ has at $x_{k-1}$ an angle smaller than $\pi$. We can fill in the triangle $\Delta_k$ as follows. We choose $x_{k}$ as point on ${(\gamma_2)}_{x_{k-1},v_3}$ at largest distance from $x_{k-1}$ such that  there exists an embedded triangle $\Delta_k$ with vertices $w_{k-1}$, $x_{k-1}$ and $x_k$ and edges $\alpha^{\sm{(k)}} = {{\gamma^{\sm{(k-1)}}}}_{w_{k-1},x_{k-1}}$ and $\beta^{\sm{(k)}} = {(\gamma_2)}_{x_{k-1},x_k}$. We denote its third edge  from $w_{k-1}$ to $x_{k}$  by $\gamma^{\sm{(k)}}$. By the triangle inequality we then have:
  \begin{equation} \label{Deltak}l(\gamma^{\sm{(k)}}) < l(\alpha^{\sm{(k)}}) + l(\beta^{\sm{(k)}}).\end{equation}
  We denote $\gamma'_{k-1} = {{\gamma^{\sm{(k-1)}}}}_{w_{k-2},w_{k-1}}$  and $\delta_{k} = \gamma'_1\cdot \ldots \cdot \gamma'_{k-1}\cdot \gamma^{\sm{(k)}} \cdot {(\gamma_2)}_{x_k,v_3}$ and obtain:
  \[\begin{array}{lcl}
    l(\delta_k) &=& l(\gamma_1') + \ldots + l(\gamma_{k-1}') + l(\gamma^{\sm{(k)}}) + l({(\gamma_2)}_{x_k,v_3}) \\
    &\stackrel{\sm (\ref{Deltak})}{<}& l(\gamma_1') + \ldots + l(\gamma_{k-1}') +  l(\alpha^{\sm{(k)}}) + l(\beta^{\sm{(k)}}) +  l({(\gamma_2)}_{x_k,v_3})\\
    &=&  l(\gamma_1') + \ldots + l(\gamma_{k-2}') +  l(\gamma^{\sm{(k-1)}}) + l({(\gamma_2)}_{x_{k-1},v_3})\\
    &=& l(\delta_{k-1}) \stackrel{\mbox{\footnotesize{induction}}}{<} l(\gamma_1) + l(\gamma_2).
    \end{array}\]
  Since the singularities lie discrete in $S$, this process must stop. Hence there exists some $k$ with $x_k = v_3$. We then have
  that $\delta = \delta_k$ is a concatenation of saddle connections from $v_1$ to $v_3$ which is homotopic to $\gamma_1\cdot \gamma_2$ and which is shorter than $\gamma_1\cdot \gamma_2$. By construction, $\delta$ is non trivial. 
\end{proof}

\begin{cor}\label{ShorteningLemma} 
  Let $c$ be a reduced closed path in $\Gamma$ of combinatorial length $\geq 2$ with an angle at one of its vertices smaller than $\pi$ and let $\gamma$ be its realization. Then there is a reduced closed edge path $\tilde{c}$ in $\Gamma$ (possibly trivial) with realization $\tilde{\gamma}$ which is shorter than $c$ and such that $\tilde{\gamma}$ is homotopic to $\gamma$.
\end{cor}

\begin{proof}
   We write the edge-path $c$ as concatenation $c_1\ldots c_n$ ($n \geq 2$) of edges and its realization $\gamma$ as
  concatenation $\gamma_1\cdot \ldots \cdot \gamma_n$ of the corresponding saddle connections.  Denote by $v_1$ the starting point of $\gamma_1$,
  by $v_2$ the end point of $\gamma_1$ which is equal to the starting point of $\gamma_2$ and by $v_3$
  the end point of $\gamma_2$. We assume without loss of generality that one of the two angles at $v_2$ is $< \pi$.\\[3mm]
It follows from Lemma~\ref{The2Case}  that there is a concatenation of saddle connections $\delta$ from $v_1$ to $v_3$ that is homotopic to $\gamma_1\cdot \gamma_2$ and shorter than $\gamma_1\cdot \gamma_2$. Let $d$ be the edge path corresponding to $\delta$ in the graph $\Gamma$, $c'$ the concatenation of $d$ and $c_3\ldots c_n$ and $\tilde{c}$ its reduction, i.e. we remove all backtracking from $c'$. Then $\tilde{c}$ has the properties required in the claim.

\end{proof}

\begin{rem}\label{AddShorter}
  With the notation of the proofs  of  Lemma~\ref{The2Case} and Corollary~\ref{ShorteningLemma} we choose $k_{\max}$ such that $x_{k_{\max}} = v_3$, i.e. we are in the situation when the process stops. Suppose that $k_{\max} \geq 2$. Let $s$ be the unique geodesic from $w_1$ to $v_2$ within the triangle $\Delta_1$ (see \textit{ Figure~\ref{fig:Deltak2}}) and $\delta$ be the concatenation $\delta = \delta_{k_{\max}} = \gamma_1' \cdot \ldots \cdot \gamma_{k_{\max}-1}' \cdot \gamma^{\sm{(k_{\max})}}$. Then we have:
  \[l(\gamma_1) + l(s) + l(\gamma_1') < l(\gamma_1) + l(\gamma_2) + l(\delta).\]
  If the reduced closed edge-path $c = c_1\ldots c_n$ with an angle  at one of its vertices smaller than $\pi$  with which we started in  Corollary~\ref{ShorteningLemma} is in addition a systole in the graph $\Gamma$, then this implies that $c = c_1c_2d^{-}$, where $d$ is the edge-path defined by $\delta$. Hence we have in this case for the realization $\gamma$ of $c$ that $\gamma = \gamma_1\cdot\gamma_2\cdot\delta^{-}$ and it is in particular null homotopic.
\end{rem}

\begin{center}
    \begin{figure}[h!]
      \begin{tikzpicture}[scale=1,<->,>=stealth',shorten >=1pt,auto,node distance=3cm,thick]
      \tikzstyle{place}=[circle,thick,draw=black!,minimum size=6mm]
      \tikzstyle{bullet}=[circle,thick,fill=black!,minimum size=2mm]
            \vertex[bullet, label = left:$v_1$](v1) at (0,0) {};
            \vertex[bullet, label = left:$v_2$](v2) at (6,8) {};
            \vertex[bullet, label = right:${v_3 = x_{k_{\max}}}$](v3) at (10,0) {};
            \node [label = {[shift={(0.0,-0.4)}]$x_1$}] (x1)  at (7.2,6.2) {};
            \node [label = {$\beta^{\sm(1)}$}] (t1)  at ($ (v2) !.7! (x1) $) {};
            \node (Delta1) at (5.5,5.5) {$\Delta_1$};
            \vertex [bullet, label = below:{$w_1$}] (w1)  at ($ (v1) !.3! (x1) $) {};
            \node [label = {[shift={(0,0.1)}]below:$\gamma_1'$}] (t5)  at ($ (v1) !.4! (w1) $) {};         
            \node  [label  = {[shift={(0.2,-.25)}]$x_2$}] (x2)  at ($ (v2) !.5! (v3) $) {};
            \node (Delta2) at (6.3,4.3) {$\Delta_2$};
             \node [label = {[shift={(.5,0)}]$\gamma^{\sm(2)}$}] (t3)  at ($ (w1) !.55! (x2) $) {};  
            \node [label = {[shift={(.5,0.2)}]$\gamma^{\sm(1)}$}] (t2)  at ($ (w1) !.2! (x1) $) {};
            \node [label = {[shift = {(-.15,0)}]below:$\alpha^{\sm(2)}$}] (t3)  at ($ (w1) !.55! (x1) $) {};
            \node [label = {[shift = {(0.2,0)}]above:$\beta^{\sm(2)}$}] (t4)  at ($ (v2) !.8! (x2) $) {};
            \vertex [bullet, label = below:{$w_2$}] (w2)  at ($ (w1) !.3! (x2) $) {};
            \node [label = {[shift={(0,0.1)}]below:$\gamma_2'$}] (t6)  at ($ (w1) !.55! (w2) $) {};
            \vertex [bullet] (w3)  at ($ (w2) !.5! (x3) $) {};
            \node  [label  = {[shift={(0.8,-.5)}]$x_{k_{\max}-1}$}] (x4)  at ($ (x3) !.6! (v3) $) {};
            \node  (delta4) [label =  {[shift={(0.35,0.4)}]$\mathbf{\delta_2}$}] at ($ (x2) !.4! (v3) $) {};

            \vertex [bullet, label = {[shift = {(-0.2,0)}]below:{$w_{k_{\max}-1}$}}] (w4)  at ($ (w3) !.4! (x4) $) {};

            \node [label = {[shift={(0,0.5)}]below:$\gamma_{k_{\max} - 1}'$}] (t8)  at (6,.75) {};

            \node [label = {[shift={(-.2,0.1)}]below:$\gamma^{\sm{(k_{\max})}}$}] (t9)  at ($ (w4) !.45! (v3) $) {};    
            \path[-, shorten >= -2mm] (w1) edge node{} (x2);
            \path[-, shorten >= -1mm] (v1) edge node{} (w1);
            \path[-, shorten <= -3mm, shorten <= -1mm ,shorten >= -2mm] (x2) edge node{} (v3);
            \path[-] (v1) edge  node {$\alpha^{\sm(1)} = \gamma_1$} (v2);
            \path[- ,shorten >= -1mm] (v2)  edge    node {} (v3);
           
            \path[- , shorten >= 1mm] (v1) edge node {} (x1);        
            
            \path[- , shorten >=  -1.5mm] (w2) edge node {} (x3);
            \path[- , shorten >= -1.5mm] (w3) edge node {} (x4);
            \path[- , shorten >= -2mm] (w4) edge node {} (v3);

            \path[-, shorten >= -1mm] (w1) edge node {} (v2);
            \node [label = {[shift = {(-.1,0)}]below:$s$}] (t15)  at ($ (w1) !.55! (v2) $) {};

            \node [label = {}] (t16) at ($ (w3) !.5! (w4) $) {};
            \path[->,  bend angle=40, bend left ]  (t8) edge node {} (t16);

          \end{tikzpicture} 
    \caption{Finding a short cut through $\Delta_1$.}
    \label{fig:Deltak2}
    \end{figure}
  \end{center}

\begin{proof}
  Using the triangle inequalities $l(s) < l(\beta^{\sm{(1)}}) + l(\alpha^{\sm{(2)}})$, $l(\alpha^{\sm{(k)}}) < l(\beta^{\sm{(k)}}) + l(\gamma^{\sm{(k)}})$ for $k \in \{1, \ldots, k_{\max}\}$ and the identities
  $\gamma^{\sm{(k-1)}} = \gamma_{k-1}' \cdot \alpha^{(k)}$ for $k \in \{2,\ldots, k_{\max - 1}\}$, we obtain:
  \[\begin{array}{lcl}
    l(s) &<& l(\beta^{\sm{(1)}}) + l(\alpha^{\sm{(2)}}) \\
    &<&  l(\beta^{\sm{(1)}}) + l(\beta^{\sm{(2)}}) +   l(\gamma^{\sm{(2)}}) \\
    &=&   l(\beta^{\sm{(1)}}) + l(\beta^{\sm{(2)}}) +   l(\gamma_2') +  l(\alpha^{\sm{(3)}}) \\
    &<& \ldots\\
    &<& l(\beta^{\sm{(1)}}) + \ldots + l(\beta^{\sm{(k_{\max})}}) +   l(\gamma_2') +  \ldots + l(\gamma_{k_{\max}-1}') + l(\gamma^{\sm{(k_{\max})}}).
  \end{array}\]
  The claimed inequality now follows from the identities \[\gamma_2 = \beta^{\sm{(1)}} \cdot \ldots \cdot \beta^{\sm{(k_{\max})}} \mbox{ and }  \delta = \gamma_1' \cdot \ldots \cdot \gamma_{k_{\max}-1}' \cdot \gamma^{\sm{(k_{\max})}}.\]
  Suppose now that in addition $c$ is a systole in the graph $\Gamma$.  Since  the (possibly trivial) closed reduced edge-path $\tilde{c}$ in $\Gamma$ that we constructed in  Corollary~\ref{ShorteningLemma} is shorter, it must be trivial.  Recall that $\tilde{c}$ is the reduction of $c'  = dc_3\ldots c_n$, where both paths $d$ and $c_3\ldots c_n$ are reduced edge-paths. Hence $\tilde{c}$ is trivial implies that $d^{-} = c_3\ldots c_n$. In this case we obtain for the original closed reduced edge-path $c$:
  \[c = c_1c_2c_3\ldots c_n = c_1c_2d^{-}. \]
\end{proof}

\begin{rem}\label{kmax1}
  In Remark~\ref{AddShorter} we require that $k_{\max} \geq 2$. What happens in the case that $k_{\max} = 1$? In this case we have that $x_1 = v_3$ and  $\delta = \gamma^{\sm{(1)}}$ is a single saddle connection which we denote by $c^{\sm{(1)}}$. The edge-path $\tilde{c}$ constructed in Corollary~\ref{ShorteningLemma} is now the reduction of $c' = c^{\sm{(1)}}c_3\ldots c_n$. If $c$ is a systole in the graph $\Gamma$, then the shorter edge-path $\tilde{c}$ must be trivial. Hence we obtain that $n = 3$ and $c^{\sm{(1)}} = c_3^{-}$ and the original edge-path $c$ is just a triangle $c = c_1c_2c_3$ of combinatorial length 3. 
\end{rem}

\begin{proof}[Proof of Theorem~\ref{n3thmN}]
  Let $c$ be a systole in the graph $\Gamma$, $n$ its combinatorial length with $n \neq 3$ and let $\gamma$ be the realization of $c$.
  We have to show that $\gamma$ is geodesic. Then it is by Remark~\ref{prop:nullhomotopicSystoles} a systole of the surface.\\
   Let us first treat the cases $n=1$ and $n = 2$. If $n = 1$, then $\gamma = \gamma_1$ is a single saddle connection and thus geodesic.
For the case $n = 2$, suppose that $\gamma = \gamma_1 \cdot \gamma_2$ is not geodesic. By Lemma~\ref{LemmaDankwart} we may assume that one of the two angles between $\gamma_1$ and $\gamma_2$ is smaller than $\pi$. Let $v$ be the start point of $\gamma_1$ which is equal to the end point of $\gamma_2$. It follows from Lemma~\ref{The2Case}  that there is a non-trivial concatenation $\delta$ of saddle connections from $v$ to $v$ which is shorter than $\gamma$. This is a contradiction to the fact that $c$ is a systole in the graph of saddle connections.\\
We now consider the case that $n \geq 4$. Suppose again that $\gamma$ is not geodesic. We are now in the situation of Corollary~\ref{ShorteningLemma}. By Remark~\ref{kmax1} it follows from the fact that the combinatorial length $n$ is greater or equal to $4$ that $k_{\max} \geq 2$. Then Remark~\ref{AddShorter} implies that $c = c_1c_2d^{-}$. We denote by $c_1'$ and $\hat{s}$ the edges in the graph $\Gamma$ corresponding to the saddle connections $\gamma_1'$ and $s$ and obtain for $c'' = c_1sc_1'$ that $l(c'') < l(\gamma_1) + l(\gamma_2) + l(\delta) = l(c)$. We know that $c''$ is a non-trivial reduced closed edge path in $\Gamma$.  Hence $c$ was not a systole in $\Gamma$ which is a contradiction.\\
If $c$ is of combinatorial length  $n=3$ the statement follows from Lemma~\ref{LemmaDankwart} and  Remark~\ref{prop:nullhomotopicSystoles}.

\end{proof}

As an immediate consequence of Theorem~\ref{n3thmN} we obtain that in the stratum $\mathcal{H}(1,1)$ of translation surfaces with two singularities of order 1 it is sufficient to study the systoles of the  graph $\Gamma$ of saddle connections in order to obtain the systolic ratio of a surface.

\begin{cor}\label{h11}
  Suppose that $S$ is a translation surface in a stratum with at most two singularities and that $\Gamma$ is its graph of saddle connections. Then  for every systole  $c$ in the graph $\Gamma$ its realization $\gamma$ is a systole on the surface. This holds in particular for the stratum $\mathcal{H}(1,1)$.
\end{cor}

\begin{proof}
 The surface $S$ has at most  two singularities. Hence $\Gamma$ has at most two vertices. Thus  a systole $c$ in $\Gamma$ has at most combinatorial length 2. Now the claim follows from Theorem~\ref{n3thmN}.
\end{proof}

\section{Systole lengths of origami surfaces}\label{section:algorithm}

A crucial role in the understanding of translation surfaces is played by \textit{origamis} or \textit{square-tiled surfaces} (cf. e.g. \cite{dgzz}, \cite{gm}, \cite{hs}, \cite{HL}, \cite{mmy}, \cite{myz}). We briefly summarize some relevant facts about origamis which are needed for this article. You can find a more elaborate introduction in \cite{W-Sch}.
Origamis are translation surfaces defined by gluing together finitely many copies of Euclidean unit squares
along the upper and lower or the right and left edges. An origami $O$ is the collection of the corresponding gluing data.
It is entirely determined by two permutations
$\sigma_a$ and $\sigma_b$ in $S_d$, where $d$ is the number of squares, $\sigma_a$ describes the horizontal gluings and $\sigma_b$
describes the vertical gluings. More precisely we label the squares by $1$, \ldots, $d$. The right edge of the square with label $i$ is glued
by a translation with the left edge of the square with label $\sigma_a(i)$. Similarly, the upper edge of the square with label $i$ is glued by a translation with the lower edge of the
square with label $\sigma_b(i)$. Renumbering the squares leads to simultaneous conjugation of the pair $(\sigma_a,\sigma_b)$. The resulting surface $S$ is then tessellated by squares and naturally carries a translation structure
(see \textit{Figure \ref{fig:ori_graph}}).
A more detailed explanation of how to
describe origamis in different combinatorial ways  can be found in Section 2 of \cite{sci1}.
The group $\SL(2,\ZZ)$ acts on the set of origamis: For $A \in \SL(2,\ZZ)$ and an origami $O$ we apply the affine map $z \mapsto Az$ to each square of $O$ (see \cite{We}, Section 2.2). We denote the corresponding translation surface by $A\cdot S$. The surface $A\cdot S$ is again an origami also denoted by $A\cdot O$, and we can determine the gluing permutations $\tilde{\sigma}_a$ and $\tilde{\sigma}_b$ of $A\cdot S$ in terms of $A$, $\sigma_a$ and $\sigma_b$ (cf. \cite{sw}, Section 2). Observe furthermore  that the geodesics in a direction $v$ on $S$ become geodesics in the direction $A\cdot v$ on $A\cdot S$.
In Lemma~\ref{sings} below we carry out the action of a system of generators of $\SL(2,\ZZ)$ fitted to the purposes of this article.\\ 

\begin{figure}[h!]
\AffixLabels{%
\centerline{%
  \begin{xy}
<0.8cm,0cm>:
(0,0)*{\OriSquare{1}{}{}{}{}};
(1,0)*{\OriSquare{2}{}{}{}{}};
(2,0)*{\OriSquare{3}{}{}{}{}};
(3,0)*{\OriSquare{4}{}{}{}{b}};
(4,0)*{\OriSquare{5}{}{}{}{c}};
(5,0)*{\OriSquare{6}{}{}{}{a}};
(0,1)*{\OriSquare{7}{}{}{}{}};
(1,1)*{\OriSquare{8}{}{}{}{}};
(2,1)*{\OriSquare{9}{}{}{}{}};
(3,1)*{\OriSquare{10}{}{a}{}{}};
(4,1)*{\OriSquare{11}{}{b}{}{}};
(5,1)*{\OriSquare{12}{}{c}{}{}};
(0,2)*{\OriSquare{13}{}{}{}{}};
(1,2)*{\OriSquare{14}{}{}{}{}};
(2,2)*{\OriSquare{15}{}{}{}{}};
(-0.5,-0.47)*{{\bullet}};
(5.5,-0.47)*{{\bullet}};
(3.5,1.53)*{{\bullet}};
(2.5,-0.47)*{\bullet};
(2.5,2.53)*{\bullet};
(-.5,2.53)*{\bullet};
(4.5,-.47)*{\circ};
(5.5,1.53)*{\circ};
(-.5,1.53)*{{\circ}};
(2.5,1.53)*{{\circ}};
\end{xy} }}
\caption{A surface $S$ defined by an origami $O$. Edges with same labels are glued.
Each edge without label is glued to the one which lies opposite to it. $S$ has two singularities $\bullet$ and $\circ$.}
\label{fig:ori_graph}
\end{figure}

It is a crucial property of origamis that the corresponding translation surfaces lie dense in the strata. Thus we may use origamis  to determine the supremum of the lengths of systoles of all translation surfaces in a stratum.  The combinatorial definition of origamis makes it easier to compute the length of their systoles explicitly. In this section we present an algorithm which determines for origamis a finite subgraph of the graph of saddle connections (cf. Definition \ref{fgos}) which contains the systoles of the full graph   (cf. Algorithm I and II). The length of a systole of this finite graph can then be computed by classical methods of graph theory.
Recall from Corollary~\ref{h11} that for translation surfaces $S$ in the stratum $\mathcal{H}(1,1)$ the systoles of the graph of saddle connections correspond to the systoles of $S$. We use this fact to compute the length of the systoles of all origamis in $\mathcal{H}(1,1)$ with up to 67 squares. In the end, we construct an explicit origami of genus $g = 3$  which shows that in other strata the length of the systoles of a translation surface can be bigger than the length of the systoles of its graph of  saddle connections. \\

Suppose from now on that $S$ is a translation surface coming from an origami $O$ consisting of $d$ squares given by two permutations $\sigma_a$ and $\sigma_b$ in $S_d$. Let  furthermore $\sy(S)$ be the length of the systoles of $S$, let $p_1 ,\ldots, p_m$ be its singularities and let $\Gamma$ be its graph of saddle connections. Observe that the developing vectors of saddle connections of an origami are integer vectors. Since  we consider  $\Gamma$ as an undirected graph, we may restrict to the primitive directions that lie in the closed upper half plane, i.e. the upper half plane including the positive $x$-axis:
\begin{equation}\label{A1}
  A_1 = \{{x\choose y} \in \ZZ^2|\; {x\choose y} \mbox{ is primitive }, y > 0 \mbox{ or } ( y = 0 \mbox{ and } x > 0)\} \,.
\end{equation}
The graph $\Gamma$ of saddle connections is infinite, however we define in Definition~\ref{fgos} several finite variants of  it which we will use to compute the systole of $\Gamma$.

\begin{defi}\label{fgos}
  We define the following subgraphs of the graph $\Gamma$ of saddle connections:
  \begin{enumerate}
  \item[i)] For a finite set $S \subseteq \ZZ^2$ let $\Gamma_S$ be the subgraph of $\Gamma$ which contains all edges of $\Gamma$ corresponding to saddle connections whose directions lie in $S$.
  \item[ii)] We consider the following special cases of i):
    \begin{itemize}
    \item For  $v \in \ZZ^2\backslash \{{0 \choose 0} \}$ we define $\Gamma_v = \Gamma_S$ with $S = \{v\}$.
    \item For $l > 0$ we define $\Gamma_{S_l} = \Gamma_S$ with $S = S_l = \{{x \choose y} \in A_1|\; x^2 + y^2 \leq l\}$.
    \end{itemize}
  \item[iii)] For $r > 0$ let $\Gamma_{\leq r}$ be the subgraph of $\Gamma$ which contains the edges of $\Gamma$ corresponding to saddle connections of length $\leq r$.
  \end{enumerate}
\end{defi}

Observe in particular that if $r \geq \mbox{sys}(S)$ then $\Gamma_{\leq r}$ and $\Gamma_{S_{r^2}}$  have the same systoles as $\Gamma$. Hence we can just compute their systoles in order to determine the length of the systoles of $\Gamma$.

\begin{ex}
  Let $S$ be the origami surface shown in \textit{Figure~\ref{fig:ori_graph}}. From this picture one can
  directly read off the two graphs $\Gamma_v$ of saddle connections for $v = {1 \choose 0}$ and $v = {0\choose 1}$ shown in Figure~\ref{graphs1}.
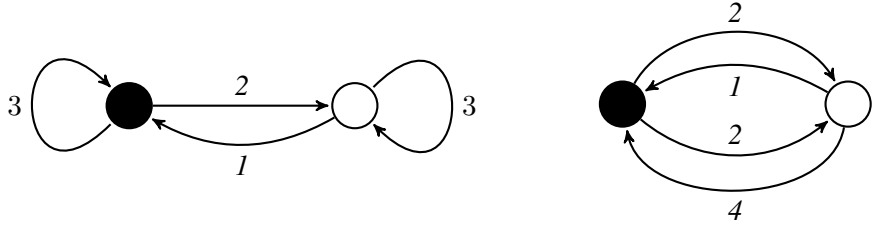
  \begin{center}
  \begin{figure}[h!]
    \AffixLabels{
          \begin{tikzpicture}[scale=1,<->,>=stealth',shorten >=1pt,auto,node distance=3cm,thick]
            \tikzstyle{place}=[circle,thick,draw=black!,minimum size=6mm]
            \tikzstyle{bullet}=[circle,thick,fill=black!,minimum size=6mm]
            \vertex[bullet](p1) at (0,0) {};
            \vertex[place](p2) at (3,0) {};
            \path[->] (p1) edge node {2} (p2);
            \path[->] (p2)  edge   [bend left] node {1} (p1);
            \Loop[dist=2cm,dir=WE,label=$3$,labelstyle=left](p1);  
            \Loop[dist=2cm,dir=EA,label=$3$,labelstyle=right](p2);  
          \end{tikzpicture}\hspace*{15mm}
          \begin{tikzpicture}[scale=1,<->,>=stealth',shorten >=1pt,auto,node distance=3cm,thick]
            \tikzstyle{place}=[circle,thick,draw=black!,minimum size=6mm]
            \tikzstyle{bullet}=[circle,thick,fill=black!,minimum size=6mm]
            \vertex[bullet](p1) at (0,0) {};
            \vertex[place](p2) at (3,0) {};
            \path[->] (p1) edge [bend left = 60] node {2} (p2);
            \path[->] (p2)  edge   [bend left = 80] node {4} (p1);
            \path[->] (p1) edge  [bend right = 40] node {2} (p2);
            \path[->] (p2)  edge   [bend right] node {1} (p1);
          \end{tikzpicture}
    }
    \caption{
      Left side: the graph $\Gamma_{v_0}$  with $v_0 = {1 \choose 0}$,\quad  Right side: the graph $\Gamma_{v_1}$  with $v_1 = {0 \choose 1}$\\
      of the origami $O$ from \textit{Figure~\ref{fig:ori_graph}.}
    }\label{graphs1}
    \end{figure}
          \end{center}
In particular, we have a closed geodesic of length 2 as concatenation of
          a vertical and a horizontal saddle connection each of length $1$. This
          gives us an upper bound $r_0 = 2$ for the length of the systole
          and the set
          \[S_{l_0} = \{{1\choose 0}, {1\choose 1}, {0\choose 1},{-1 \choose 1}\} \mbox{ with } l_0 = r_0^2 = 4.\]
          For $v_2 = {1\choose 1}$ and $v_3 = {-1 \choose 1}$ we obtain the two graphs $\Gamma_v$ of saddle connections shown in Figure~\ref{g2}.
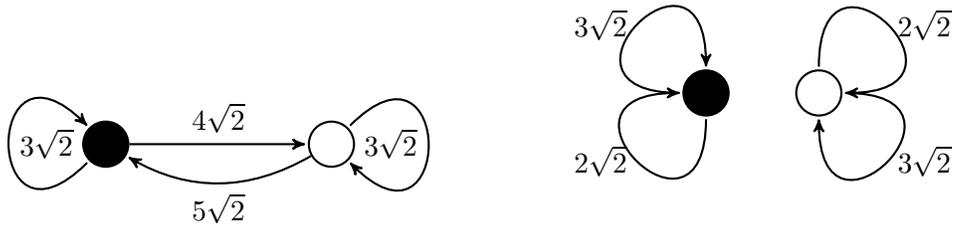
          \begin{center}
            \begin{figure}[h!]
              \AffixLabels{
                \begin{tikzpicture}[scale=1,<->,>=stealth',shorten >=1pt,auto,node distance=3cm,thick]
                  \tikzstyle{place}=[circle,thick,draw=black!,minimum size=6mm]
                  \tikzstyle{bullet}=[circle,thick,fill=black!,minimum size=6mm]
                  \vertex[bullet](p1) at (0,0) {};
                  \vertex[place](p2) at (3,0) {};
                  \path[->] (p1) edge node {$4\sqrt{2}$} (p2);
                  \path[->] (p2)  edge   [bend left] node {$5\sqrt{2}$} (p1);
                  \Loop[dist=2cm,dir=WE,label=$3\sqrt{2}$,labelstyle=right](p1);  
                  \Loop[dist=2cm,dir=EA,label=$3\sqrt{2}$,labelstyle=left](p2);  
                \end{tikzpicture}\hspace*{10mm}
                \begin{tikzpicture}[scale=1,<->,>=stealth',shorten >=1pt,auto,node distance=3cm,thick]
                  \tikzstyle{place}=[circle,thick,draw=black!,minimum size=6mm]
                  \tikzstyle{bullet}=[circle,thick,fill=black!,minimum size=6mm]
                  \vertex[bullet](p1) at (0,0) {};
                  \vertex[place](p2) at (1.5,0) {};
                  \Loop[dist=2cm,dir=NOWE,label=$3\sqrt{2}\,$,labelstyle=left](p1);  
                  \Loop[dist=2cm,dir=NOEA,label=$\,2\sqrt{2}$,labelstyle=right](p2);
                  \Loop[dist=2cm,dir=SOWE,label=$2\sqrt{2}\,\,$,labelstyle=left](p1);  
                  \Loop[dist=2cm,dir=SOEA,label=$3\sqrt{2}$,labelstyle=right](p2);                    
                \end{tikzpicture}
                }
                \caption{
                   Left side: the graph $\Gamma_{v_2}$  with $v_2 = {1 \choose 1}$,\quad  Right side: the graph $\Gamma_{v_3}$  with $v_3 = {-1 \choose 1}$\\
                   of the origami $O$ from \textit{Figure~\ref{fig:ori_graph}.
                   }\label{g2}
                 }
            \end{figure}
          \end{center}
          Hence $\Gamma_{S_{l_0}}$ becomes the undirected graph shown in Figure~\ref{g3}.
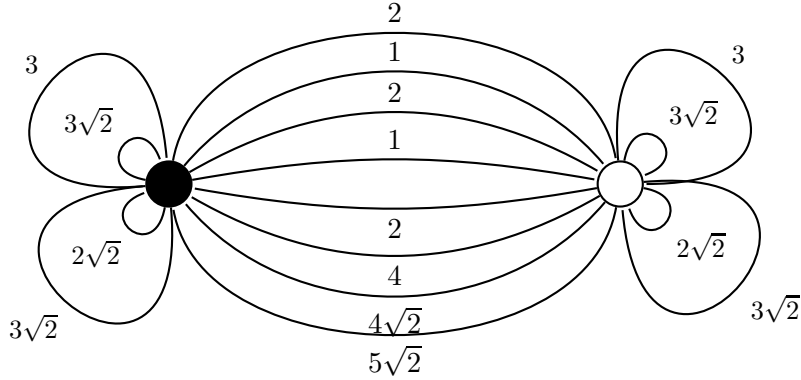
          \begin{center}
            \begin{figure}[h!]\label{graphs2}
              \AffixLabels{
                \begin{tikzpicture}[scale=1,-,>=stealth',shorten >=1pt,auto,node distance=3cm,thick, every loop/.style={}]]
              \tikzstyle{place}=[circle,thick,draw=black!,minimum size=6mm]
              \tikzstyle{bullet}=[circle,thick,fill=black!,minimum size=6mm]
              \vertex[bullet](p1) at (0,0) {};
              \vertex[place](p2) at (6,0) {};
              \path[every node/.style={font=\sffamily\small}] (p1) edge [bend left = 80] node {$2$}  (p2);
              \path[every node/.style={font=\sffamily\small}] (p1) edge [bend left = 50] node {$1$} (p2);
              \path[every node/.style={font=\sffamily\small}] (p1) edge [bend left = 30] node {$2$} (p2);
              \path[every node/.style={font=\sffamily\small}] (p1) edge [bend left = 10] node {$1$} (p2);
              \path[every node/.style={font=\sffamily\small}] (p2) edge [bend left = 80] node[yshift=3.5pt] {$5\sqrt{2}$}  (p1);
              \path[every node/.style={font=\sffamily\small}] (p2) edge [bend left = 50] node[yshift=3.5pt]{$4\sqrt{2}$} (p1);
              \path[every node/.style={font=\sffamily\small}] (p2) edge [bend left = 30] node {$4$} (p1);
              \path[every node/.style={font=\sffamily\small}] (p2) edge [bend left = 10] node {$2$} (p1);
              \path[every node/.style={font=\sffamily\small}] (p1) edge[loop, out=185, in=95, looseness=0.8,distance=3.5cm]  node{$3$} (p1);
              \path[every node/.style={font=\sffamily\small}] (p1) edge[loop, out=170, in=110, looseness=0.8,distance=.8cm]  node{$3\sqrt{2}$} (p1);
              \path[every node/.style={font=\sffamily\small}] (p1) edge[loop, out=275, in=185, looseness=0.8,distance=3.5cm]  node{$3\sqrt{2}$} (p1);
              \path[every node/.style={font=\sffamily\small}] (p1) edge[loop, out=260, in=200, looseness=0.8,distance=.8cm]  node{$2\sqrt{2}$} (p1);
              \path[every node/.style={font=\sffamily\small}] (p2) edge[loop, out=95, in=0, looseness=0.8,distance=3.5cm]  node{$3$} (p2);
              \path[every node/.style={font=\sffamily\small}] (p2) edge[loop, out=80, in=20, looseness=0.8,distance=.8cm]  node{$3\sqrt{2}$} (p2);
              \path[every node/.style={font=\sffamily\small}] (p2) edge[loop, out=365, in=275, looseness=0.8,distance=3.5cm]  node{$3\sqrt{2}$} (p2);
              \path[every node/.style={font=\sffamily\small}] (p2) edge[loop, out=350, in=290, looseness=0.8,distance=0.8cm]  node{$2\sqrt{2}$} (p2);
                \end{tikzpicture}\hspace*{10mm}
                 }
                \caption{
                   The graph $\Gamma_{S_{l_0}}$ with $l_0 = 4$ of the origami $O$ from \textit{Figure~\ref{fig:ori_graph}.}
                 }\label{g3}
            \end{figure}
          \end{center} 
        The length of a smallest closed reduced path in $\Gamma_{S_{l_0}}$ is $2$, hence the length of the systole of the graph $\Gamma$ of all saddle connections  is equal to $2$. In particular,
        we observe that in this example there is no regular closed geodesic on the translation surface which realizes the length of the systoles.
          \end{ex}

From the considerations above we obtain the following ansatz for the computation of the length of the systoles of the graph $\Gamma$ of saddle connections:
\begin{itemize}
\item
  We find an upper bound $r$  for the length of the systoles of $\Gamma$. We take here the minimum of the lengths of all regular horizontal and vertical closed geodesics on $S$.
\item
  For each $v$ in $S_{r^2}$  we compute the graph $\Gamma_v$ of saddle connections in directions $v$.
\item
  We obtain the graph $\Gamma_{S_{r^2}}$ as the union of all $\Gamma_v$ with $v \in S_{r^2}$.
\end{itemize}
The systoles of the finite subgraph $\Gamma_{S_{r^2}}$ of $\Gamma$ then equal the systoles of $\Gamma$.\\

We record in Algorithm II below how to build  $\Gamma_v$ for a given $v \in A_1$. Using this we obtain from the ansatz described above  Algorithm I for the computation of the systoles of the graph $\Gamma$ of saddle connections.\\

\begin{tcolorbox}
  \textbf{Algorithm I \textit{Computation of the graph \boldmath{$\Gamma_{S}$} of an origami \boldmath{$O$}}}\\[3mm]
  Suppose that the origami $O$ is given by the pair of permutations $(\sigma_a,\sigma_b)$.
  \begin{enumerate}[label=\protect\circled{\arabic*}]
  \item
    Compute the minimal length $r_0$ of all horizontal and vertical regular closed geodesics as follows:
    \[ r_0 = \min\{\mbox{lengths of cycles in the permutation } \sigma_a \mbox{ and } \sigma_b\}\]
    It follows that $\sy(S) \leq r_0$.
  \item
    Let $l_0 := r_0^2$ and $S := S_{l_0}$ (cf. Definition \ref{fgos}).\\  
    For each $v \in S$ calculate the weighted graph $\Gamma_v$ 
    of saddle connections in direction $v$ (see Algorithm II). Hence each $\Gamma_v$ is a graph with vertices labeled
    by the singularities $p_1$, \ldots, $p_m$ of the origami.
  \item
    Compute the graph $\Gamma_S$ as union of the graphs $\Gamma_v$ over $v \in S$.
    More precisely: $\Gamma_S$ again has $m$ vertices labeled by $p_1$, \ldots, $p_m$.
    For all $v \in S$ and for all edges $e$
    between $p_i$ and $p_j$  in $\Gamma_v$ include an edge between $p_i$ and $p_j$ in $\Gamma_S$ with the 
    same weight that the edge $e$ has in $\Gamma_v$. 
  \end{enumerate}
  For each shortest closed edge-path $c$ whose realization $\gamma$ on $S$ is not null homotopic, $\gamma$ is then a systole of the origami.
\end{tcolorbox}

\begin{note}
  You can improve Algorithm I by taking $r_0$ as the length of a shortest closed
  reduced path in the graph $\Gamma_v$ with $v \in \{{1 \choose 0}, {0 \choose 1}\}$ or any other upper bound for
  the length of the systole.
\end{note}

It remains to describe, how to compute  $\Gamma_v$ for a given $v \in A_1$ (see Algorithm II).\\
Firstly, we consider the case $v_0 = {1 \choose 0}$. In this case we obtain  $\Gamma_{v_0}$ directly from the permutations $(\sigma_a,\sigma_b)$ of the origami $O$, as follows.  Recall that the singularities of an origami $O$ given by the pair of permutations ($\sigma_a$, $\sigma_b$) are in one-to-one correspondence with the cycles  of the commutator $\sigma_c = [\sigma_a,\sigma_b] = \sigma_a\circ \sigma_b \circ \sigma_a^{-1} \circ \sigma_b^{-1}$. More precisely, for each cycle $(i_1, \ldots, i_k)$ of $\sigma_c$ we have that the left lower corner of the squares labeled by $i_1,\ldots, i_k$ belong to the same singularity. For $i$ we denote the cycle which contains $i$ as well as the corresponding singularity by $[i]$. We make a list $L$ of all squares whose left lower corner is a singularity, namely all $i$ such that $\sigma_c(i) \neq i$. For each element $i$ in $L$ we choose $k$ minimal such that $\sigma_a^k(i)$ is in $L$. This gives us a horizontal saddle connection in $\Gamma_{v_0}$ of length $k$ starting in the singularity corresponding to the cycle $[i]$.\\
Now, we turn to the graph $\Gamma_v$ for an arbitrary element $v$ in $A_1$. Since $v$ is primitive in $\ZZ^2$ we may choose some $A \in \SL(2,\ZZ)$ with $A\cdot v = v_0$. We consider the origami $O_A = A\cdot O$ and compute its permutations $(\sigma'_a$, $\sigma'_b)$. We further use that every saddle connection in direction $v$ of length $l$ of the origami $O$ becomes a saddle connection of direction $A\cdot v = v_0$ of the origami $O_A$ of length $l_A = l \cdot ||v||$. Hence the desired graph $\Gamma_v$ for the origami $O$ equals the graph $\Gamma_{v_0}$ for the origami $A\cdot O$. Only the weights which represent the lengths of the saddle connections  have to be adjusted: We have to multiply the weights in $\Gamma_{v_0}$ (of $A\cdot O$) by $\frac{||v||}{||A\cdot v||}$ to obtain the weights in $\Gamma_v$ (of $O$).\\
As last step, we have to identify which vertex of the graph $\Gamma_v$ corresponds to which singularity. More precisely, we have to identify the singularities of $O_A = A\cdot O$ with the singularities of $O$. The identification is described in Lemma~\ref{sings} for the following generators of $\SL(2,\ZZ)$ and their inverses:
\[T = \begin{pmatrix} 1&1\\0&1 \end{pmatrix}, R = \begin{pmatrix} 0 & -1\\ 1 & 0 \end{pmatrix}, T^{-1} = \begin{pmatrix} 1&-1\\0&1 \end{pmatrix} \mbox{ and } R^{-1} = \begin{pmatrix} 0 & 1\\ -1 & 0 \end{pmatrix}\]
A general matrix $A \in \SL(2,\ZZ)$ can be written as product of these four matrices and we can iterate the identifications in Lemma~\ref{sings} in order to identify the singularities of $O$ and $A\cdot O$.

\begin{lem}\label{sings}
  Let $O$ be an origami given by the pair of permutations $(\sigma_a,\sigma_b)$. For $A \in \{T,R,T^{-1}, R^{-1}\}$ the following holds.
  \begin{enumerate}
  \item[i)]
    We can enumerate the squares forming  $A\cdot O$ such that $A\cdot O$ is given by the pair of permutations:
    \[
    \begin{array}{ll}
      (\sigma_a,\sigma_b\sigma_a^{-1}) &\mbox{, if } A = T\\
      (\sigma_b^{-1},\sigma_a) &\mbox{, if } A = R\\
      (\sigma_a,\sigma_b\sigma_a) &\mbox{, if } A = T^{-1} \mbox{ and }\\
       (\sigma_b,\sigma_a^{-1}) &\mbox{, if } A = R^{-1}
    \end{array}
    \]
  \item[ii)]
    For the translation surfaces $S$ and $A\cdot S$ defined by the origamis $O$ and $A\cdot O$ we have:
    There is a natural homeomorphism $f_A$ between $S$ and $A\cdot S$ which maps geodesics in direction $v \in \RR^2$ to geodesics in direction $A\cdot v$. With respect to the enumeration of the squares forming $O$ and $A\cdot O$ fixed  in i) we have:
    The singularity $[i]$ on $S$ corresponding to the left lower corner of the square labeled by $i$ is mapped to the  singularity $[j]$ on $A\cdot S$ with:
    \[
    \begin{array}{ll}
      [j] = [i]                &\mbox{, if } A = T\\
      {[j]} = [{\sigma_b}^{-1}(i)]  &\mbox{, if } A = R\\
       {[j]} = {[i]}                &\mbox{, if } A = T^{-1} \mbox{ and }\\
      {[j]} = [\sigma_a^{-1}(i)]  &\mbox{, if } A = R^{-1}
    \end{array}
    \]
  \end{enumerate}
\end{lem}

\begin{center}
            \begin{figure}[h!]
              \AffixLabels{
                \begin{tikzpicture}[scale=1.1,-,>=stealth',shorten >=1pt,auto,node distance=3cm,thick]
                  \tikzstyle{place}=[circle,thick,draw=black!,minimum size=6mm]
                  \tikzstyle{bullet}=[circle,thick,fill=black!,minimum size=1mm]
                  \draw  (-1,0) -- (0,0) -- (1,0) -- (2,0);
                  \draw  (-1,1) -- (0,1) -- (1,1) -- (2,1);
                  \draw (-1,2) -- (0,2);
                  \draw (-1,0) -- (-1,1); \draw (0,0) -- (0,1); \draw (1,0) -- (1,1);
                  \draw (-1,1) -- (-1,2); \draw (0,1) -- (0,2);
                  \draw (2,0) -- (2,1);
                  \node at (.5,.5) (y1) {${\sm i}$};
                  \node at (-.5,.5) (y2){${\sm \sigma_a^{-1}(i)}$};
                  \node at (-.5,1.5) (y3){${\sm \sigma_b(\sigma_a^{-1}(i))}$};
                   \node at (1.5,.5) (y4) {${\sm \sigma_a(i)}$};
                  \draw[->] (2.5,1) -- (4,1) node[midway, above] {$f_T$};;
                  \vertex[bullet, label = below:$s$]() at (0,0) {};
                \end{tikzpicture}\hspace*{1mm}
                \begin{tikzpicture}[scale=1.2,-,>=stealth',shorten >=1pt,auto,node distance=3cm,thick]
                  \tikzstyle{place}=[circle,thick,draw=black!,minimum size=6mm]
                  \tikzstyle{bullet}=[circle,thick,fill=black!,minimum size=1mm]
                  \draw  (-1,0) -- (0,0) -- (1,0) -- (2,0);
                  \draw  (0,1) -- (1,1) -- (2,1) -- (3,1);
                  \draw (1,2) -- (2,2);
                  \draw (-1,0) -- (0,1); \draw (0,0) -- (1,1); \draw (1,0) -- (2,1);
                  \draw (0,1) -- (1,2); \draw (1,1) -- (2,2);
                  \draw (2,0) -- (3,1);
                  \node at (1,.5) (y1) {${\sm i}$};
                  \node at (0,.5) (y2){${\sm \sigma_a^{-1}(i)}$};
                  \node at (.9,1.4) (y3){${\sm \sigma_b(\sigma_a^{-1}(i))}$};
                   \node at (2,.5) (y4){${\sm \sigma_a(i)}$};
                   \vertex[bullet, label = below:$f_T(s)$]() at (0,0) {};
                  \node at (3.2,.5) {${=\joinrel=}$};
                \end{tikzpicture}\hspace*{-1mm}
                \begin{tikzpicture}[scale=1.2,-,>=stealth',shorten >=1pt,auto,node distance=3cm,thick]
                  \tikzstyle{place}=[circle,thick,draw=black!,minimum size=6mm]
                  \tikzstyle{bullet}=[circle,thick,fill=black!,minimum size=1mm]
                  \draw  (-1,0) -- (0,0) -- (1,0) -- (2,0);
                  \draw  (0,1) -- (1,1) -- (2,1) -- (3,1);
                  \draw (1,2) -- (2,2);
                  \draw (-1,0) -- (0,1); \draw (0,0) -- (1,1); \draw (1,0) -- (2,1);
                  \draw (0,1) -- (1,2); \draw (1,1) -- (2,2);
                  \draw (2,0) -- (3,1);
                  \draw[dashed] (0,0) -- (0,1); \draw[dashed] (1,0) -- (1,1); \draw[dashed] (1,1) -- (1,2); \draw[dashed] (2,0) -- (2,1);
                  \node at (.5,.2) (y1) {${\sm i}$};
                  \node at (-.5,.2) (y2){${\sm \sigma_a^{-1}(i)}$};
                  \node at (-.4,1.4) (y3){${\sm \sigma_b(\sigma_a^{-1}(i))}$};
                  \node at (1.6,.2) (y4){${\sm \sigma_a(i)}$};
                  \draw [->, bend left](.1,1.4) to (.6,1.1);
                   \vertex[bullet, label = below:$f_T(s)$]() at (0,0) {};
                \end{tikzpicture}                
              }
              \caption{Local picture of the application of $T$ to an origami}\label{happyShear}
            \end{figure}
\end{center}
\begin{proof}
  Figure~\ref{happyShear} shows the application of the matrix $T$. More precisely, it shows the three squares labeled by $i$, $\sigma_a^{-1}(i)$ and ${\sigma_b}^{-1}(\sigma_a^{-1}(i))$ on the surface $S$  and their images on the surface $T\cdot S$. The surface $T\cdot S$ is tiled by parallelograms. The labeling of the squares on $S$ carries over to a labeling of the parallelograms on $T\cdot S$. We consider the map $f_T: S \to T\cdot S$ which maps each square on $S$ to the corresponding parallelogram on $T\cdot S$. Now, we explicitly obtain a square tiling of $T \cdot S$ using the horizontal edges of the parallelograms and adding as vertical edges the diagonals of the parallelograms as shown in Figure~\ref{happyShear} on the right side. We label each square with the number of the parallelogram which intersects its right lower part. In this way we obtain the permutations stated in i). The vertex $s$ which is the left lower vertex of the square on $S$ labeled with $i$ is mapped on $T\cdot S$ to the left lower vertex of the square labeled again with $i$. This proves the statement of ii) with regard to the action of $T$. The other three cases follow in a similar way as  shown in Figure~\ref{ApplyS} and \ref{ApplyTInverse}.
\end{proof}

\begin{center}
            \begin{figure}[h!]
              \AffixLabels{
                 \begin{minipage}[b]{3cm}
                \begin{tikzpicture}[scale=1.2,-,>=stealth',shorten >=1pt,auto,node distance=3cm,thick]
                  \tikzstyle{place}=[circle,thick,draw=black!,minimum size=6mm]
                  \tikzstyle{bullet}=[circle,thick,fill=black!,minimum size=1mm]
                  \draw  (0,0) -- (1,0) -- (2,0);
                  \draw  (0,-1) -- (1,-1);
                  \draw (0,1) -- (1,1) -- (2,1);
                  \draw (0,-1) -- (0,0); \draw (1,-1) -- (1,0);
                  \draw (0,0) -- (0,1); \draw (1,0) -- (1,1); \draw (2,0) -- (2,1);
                  \node at (.5,.5) (y1) {${\sm i}$};
                  \node at (.5,-.5) (y2){${\sm \sigma_b^{-1}(i)}$};
                  \node at (1.5,.5) (y3){${\sm \sigma_a(i)}$};
                  \draw[->] (2.5,0) -- (3.5,0) node[midway, above] {$f_R$};;
                  \vertex[bullet, label = left:${s}$]() at (0,0) {};
                \end{tikzpicture}
                \end{minipage}\hspace*{20mm}
                 \begin{minipage}[b]{3cm}
                \begin{tikzpicture}[scale=1.2,-,>=stealth',shorten >=1pt,auto,node distance=3cm,thick]
                  \tikzstyle{place}=[circle,thick,draw=black!,minimum size=6mm]
                  \tikzstyle{bullet}=[circle,thick,fill=black!,minimum size=1mm]
                  \draw  (0,0) -- (1,0) -- (2,0);
                  \draw  (0,2) -- (1,2);
                  \draw (0,1) -- (1,1) -- (2,1);
                  \draw (0,1) -- (0,2); \draw (1,1) -- (1,2);
                  \draw (0,0) -- (0,1); \draw (1,0) -- (1,1); \draw (2,0) -- (2,1);
                  \node at (.5,.5) (y1) {${\sm i}$};
                  \node at (1.5,.5) (y2){${\sm \sigma_b^{-1}(i)}$};
                  \node at (.5,1.5) (y3){${\sm \sigma_a(i)}$};
                  \vertex[bullet, label = below:${f_{R}(s)}$]() at (1,0) {};
                \end{tikzpicture}
                 \end{minipage}\hspace*{10mm}
                \begin{minipage}[b]{3cm}
                 \begin{tikzpicture}[scale=1.2,-,>=stealth',shorten >=1pt,auto,node distance=3cm,thick]
                  \tikzstyle{place}=[circle,thick,draw=black!,minimum size=6mm]
                  \tikzstyle{bullet}=[circle,thick,fill=black!,minimum size=1mm]
                  \draw  (0,0) -- (1,0) -- (2,0);
                  \draw  (1,2) -- (2,2);
                  \draw (0,1) -- (1,1) -- (2,1);
                  \draw (1,1) -- (1,2); \draw (2,1) -- (2,2);
                  \draw (0,0) -- (0,1); \draw (1,0) -- (1,1); \draw (2,0) -- (2,1);
                  \node at (.5,.5) (y1) {${\sm \sigma_a^{-1}(i)}$};
                  \node at (1.5,1.5) (y2){${\sm \sigma_b(i)}$};
                  \node at (1.5,.5) (y3){${\sm i}$};
                  \draw[->] (2.5,1) -- (3.5,1) node[midway, above] {$f_{R^{-1}}$};;
                  \vertex[bullet, label = below:$s$]() at (1,0) {};
                 \end{tikzpicture}\hspace*{1.5mm}
                  \end{minipage}
                  \begin{minipage}[b]{3cm}
                \begin{tikzpicture}[scale=1.2,-,>=stealth',shorten >=1pt,auto,node distance=3cm,thick]
                  \tikzstyle{place}=[circle,thick,draw=black!,minimum size=6mm]
                  \tikzstyle{bullet}=[circle,thick,fill=black!,minimum size=1mm]
                  \draw  (0,0) -- (1,0) -- (2,0);
                  \draw  (0,2) -- (1,2);
                  \draw (0,1) -- (1,1) -- (2,1);
                  \draw (0,1) -- (0,2); \draw (1,1) -- (1,2);
                  \draw (0,0) -- (0,1); \draw (1,0) -- (1,1); \draw (2,0) -- (2,1);
                  \node at (.5,.5) (y1) {${\sm i}$};
                  \node at (1.5,.5) (y2){${\sm \sigma_b(i)}$};
                  \node at (.5,1.5) (y3){${\sm \sigma_a^{-1}(i)}$};
                  \vertex[bullet, label = below left:$f_{R^{-1}}(s)$]() at (0,1) {};
                \end{tikzpicture}
                 \end{minipage}
              }
              \caption{Local pictures of the application of $R$ and $R^{-1}$ to an origami}\label{ApplyS}
            \end{figure}
            \begin{figure}[h!]
              \AffixLabels{ 
                \begin{tikzpicture}[scale=1.2,-,>=stealth',shorten >=1pt,auto,node distance=3cm,thick]
                  \tikzstyle{place}=[circle,thick,draw=black!,minimum size=6mm]
                  \tikzstyle{bullet}=[circle,thick,fill=black!,minimum size=1mm]
                  \draw  (-1,0) -- (0,0) -- (1,0);
                  \draw  (-1,1) -- (0,1) -- (1,1);
                  \draw (0,2) -- (1,2);
                  \draw (-1,0) -- (-1,1); \draw (0,0) -- (0,1); \draw (1,0) -- (1,1);
                  \draw (0,1) -- (0,2); \draw (1,1) -- (1,2);
                  \node at (-.5,.5) (y1) {${\sm i}$};
                  \node at (.5,.5) (y2){${\sm \sigma_a(i)}$};
                  \node at (.5,1.5) (y3){${\sm \sigma_b(\sigma_a(i))}$};
                  \draw[->] (1.5,1) -- (2.5,1) node[midway, above] {$f_{T^{-1}}$};;
                  \vertex[bullet, label = below:$s$]() at (-1,0) {};
                 \end{tikzpicture}\hspace*{1mm}
                                 \begin{tikzpicture}[scale=1.2,-,>=stealth',shorten >=1pt,auto,node distance=3cm,thick]
                  \tikzstyle{place}=[circle,thick,draw=black!,minimum size=6mm]
                  \tikzstyle{bullet}=[circle,thick,fill=black!,minimum size=1mm]
                  \draw  (-1,0) -- (0,0) -- (1,0);
                  \draw  (-2,1) -- (-1,1) -- (0,1);
                  \draw (-2,2) -- (-1,2);
                  \draw (-1,0) -- (-2,1); \draw (0,0) -- (-1,1); \draw (1,0) -- (0,1);
                  \draw (-1,1) -- (-2,2); \draw (0,1) -- (-1,2);
                  \draw[dashed] (0,0) -- (0,1); \draw[dashed] (-1,0) -- (-1,1); \draw[dashed] (-1,1) -- (-1,2);
                  \node at (.5,.2) (y1) {${\sm \sigma_a(i)}$};
                  \node at (-.5,.2) (y2){${\sm i}$};
                  \node at (-.6,1.2) (y3){${\sm \sigma_b(\sigma_a(i))}$};
                   \vertex[bullet, label = below:$f_{T^{-1}}(s)$]() at (-1,0) {};
                \end{tikzpicture}              
              }
              \caption{Local picture of the application of $T^{-1}$ to an origami}\label{ApplyTInverse}
            \end{figure}
\end{center}

\begin{tcolorbox}
  \textbf{Algorithm II \textit{Computation of the graph \boldmath{$\Gamma_v$}}}\\[3mm]
  Suppose we are given: the origami $O$ by the pair of permutations $(\sigma_a,\sigma_b)$  and a direction $v \in A_1$.\\
  Recall that the singularities $p_1$, \ldots, $p_m$ of the origami $O$ are identified with the cycles $[i]$ of the commutator $[\sigma_a,\sigma_b] = \sigma_a\circ\sigma_b\circ\sigma_a^{-1}\circ\sigma_b^{-1}$.
\begin{enumerate}[label=\protect\circled{\arabic*}]
  \item
    Choose a matrix $A \in \SL_2(\Z)$ such that $A\cdot v = {1\choose 0}$.\\
    Write $A$ as product $A = A_1\cdot \ldots \cdot A_k$ with $A_i \in \{T,R,T^{-1}, R^{-1}\}$.\\
    Compute iteratively the pair of permutations $(\sigma'_a, \sigma'_b)$ for the origami $O_A = A\cdot O$ using Lemma~\ref{sings} i).\\
    Compute iteratively the bijection $\alpha:[i] \mapsto [\alpha(i)]$ between the singularities of $O$ and those of $O_A$ using Lemma~\ref{sings} ii).
  \item
    The vertices of $\Gamma_v$ are the cycles of the permutation $[\sigma_a',\sigma_b']$ of length $\geq 2$.\\ Recall that the cycle $[j]$ of $[\sigma_a',\sigma_b']$ corresponds to the cycle $[\alpha^{-1}(j)]$ of $[\sigma_a,\sigma_b]$. \\
    Label now the vertex $[j]$ with the singularity $p_k$ $(k \in \{1,\ldots,m\})$ on $O$ which corresponds to the cycle $[\alpha^{-1}(j)]$.
  \item
    Make a list $L$ of those squares of the 
    origami $O_A$ whose left lower corners are 
    singularities, i.e.  $L = \{i \in \{1 ,\ldots, n\}|\;\sigma_c'(i) \neq i\}$, where $\sigma_c' = [\sigma_a',\sigma_b']$ .
  \item
    For each $i \in L$ do:\\
    Let $k \geq 1$ be the 
    smallest integer number such that $j = ({\sigma'}_a)^k(i) \in L$.\\
    Put a directed edge labeled by $k \cdot \ell(v)$ from the singularity
    which corresponds to the cycle of $[\sigma_a',\sigma_b']$ containing $i$ to
    the singularity which corresponds to the cycle containing $j$.
\end{enumerate}
\end{tcolorbox}

\vspace*{5mm}

Recall from Corollary \ref{h11} that for surfaces $S$ in the stratum $\mathcal{H}(1,1)$ the length of the systoles of $S$ is indeed equal to the length of the systoles of the graph of saddle connections. Thus we can use Algorithm~I and Algorithm~II in order to study  the systoles of translations of surfaces of origamis in $\mathcal{H}(1,1)$. Using the programs from \cite{Co} we obtain \textit{Table}~\ref{H11}. The table shows for given $n$ the maximal length $\sy$ of a systole in the graph of saddle connections and the corresponding maximal systolic ratio $\sr = \sy^2/n$  that is achieved by an origami with $n$ squares in $\mathcal{H}(1,1)$. The minimal number of squares for an origami in this stratum is $4$. In particular, we obtain from \textit{Table}~\ref{H11} that among the origamis of degree $n \leq 67$ the maximal systolic ratio for closed paths in the graph of saddle connections is $\frac{17}{30} \sim 0.567$. This supports Conjecture 1.2 \cite{jp} of Judge and Parlier that the maximal systolic ratio in $\mathcal{H}(1,1)$ is $0.58404...$.\\

  \begin{table}
    \centering
    \begin{tabular}[t]{rcr}
      \hline
      $n$   &   $\sy$ & $\sr$  \\
      \hline
      $4$   & $\sqrt{2}$  & $0.5$             \\
      $5$   & $\sqrt{2}$  & $0.4$             \\
      $6$   & $\sqrt{2}$  & $\sim 0.33$       \\
      $7$   & $\sqrt{2}$  & $\sim 0.29$       \\
      $8$   & $2$         & $0.5$             \\
      $9$   & $2$         & $\sim 0.44$       \\
      $10$  & $2$         & $0.4$             \\
      $11$  & $\sqrt{5}$  & $\sim 0.45$       \\
      $12$  & $\sqrt{5}$  & $\sim 0.42$       \\
      $13$  & $\sqrt{5}$  & $\sim 0.38$       \\
      $14$  & $\sqrt{5}$  & $\sim 0.36$       \\
      $15$  & $\sqrt{5}$  & $\sim 0.33$       \\
      $16$  & $1 + \sqrt{2}$  & $\sim 0.36$   \\
       $17$  & $2\,\sqrt{2}$ & $\sim 0.47$     \\
      $18$  & $2\,\sqrt{2}$ & $\sim 0.44$     \\
      $19$  & $3$           & $\sim 0.47$     \\
      $20$  & $\sqrt{10}$   & $0.50$           \\
      $21$  & $\sqrt{10}$   & $\sim 0.48$     \\
      $22$  & $\sqrt{10}$   & $\sim 0.45$     \\
      $23$    & $\sqrt{10}$    & $\sim 0.43$  \\
      $24$    & $\sqrt{13}$    & $\sim 0.54$  \\
    \end{tabular}
    \hfil
    \begin{tabular}[t]{rcr}
      \hline
      $n$   & $\sy$      & $\sr$  \\
      \hline
      $25$    & $\sqrt{13}$    & $\sim 0.52$       \\
      $26$    & $\sqrt{13}$    & $0.50$        \\
      $27$    & $\sqrt{13}$    & $\sim 0.48$ \\
      $28$    & $\sqrt{13}$    & $\sim 0.46$ \\
      $29$     & $\sqrt{13}$    & $\sim 0.45$ \\
      $30$    & $\sqrt{17}$    & $\sim 0.57$ \\
      $31$    & $4$   & $\sim 0.52$ \\
      $32$    & $\sqrt{17}$    & $\sim 0.53$\\
         $33$    & $\sqrt{17}$    & $\sim 0.52$\\
      $34$    & $\sqrt{17}$    & $0.50$          \\
      $35$    & $\sqrt{17}$    & $\sim 0.49$ \\
      $36$    & $3\,\sqrt{2}$   & $0.50$          \\
      $37$    & $\sqrt{5} + 2$ & $\sim 0.48$ \\
      $38$    & $3\,\sqrt{2}$   & $\sim 0.47$ \\
      $39$    & $3\,\sqrt{2}$   & $\sim 0.46$ \\
      $40$    & $2\,\sqrt{5}$   & $0.5$          \\
      $41$    & $2\,\sqrt{5}$   & $\sim 0.49$ \\
      $42$    & $2\,\sqrt{5}$   & $\sim 0.48$ \\
      $43$     & $2\,\sqrt{5}$   & $\sim 0.47$\\
      $44$    & $2\,\sqrt{5}$   & $\sim 0.45$ \\
      $45$    & $2\,\sqrt{5}$   & $\sim 0.44$ \\
    \end{tabular}
    \hfil
    \begin{tabular}[t]{rcr}
      \hline
      $n$   & $\sy$      & $\sr$ \\
      \hline
      $46$    & $5$   & $\sim 0.54$ \\
      $47$    & $5$   & $\sim 0.53$ \\
      $48$    & $5$   & $\sim 0.52$ \\
            $49$    & $5$   & $\sim 0.51$ \\
      $50$  &  $\sqrt{26}$  &  $0.52$  \\
      $51$  &  $\sqrt{26}$  &  $\sim 0.51$  \\
      $52$  &  $\sqrt{26}$  &  $0.50$  \\
      $53$  &  $1 + \sqrt{17}$  &  $\sim 0.50$  \\
      $54$  &  $2 + \sqrt{10}$  &  $\sim 0.49$  \\
      $55$  &  $1 + 3\,\sqrt{2}$  &  $\sim 0.50$  \\
      $56$  &  $\sqrt{29}$  &  $\sim 0.52$  \\
      $57$  &  $\sqrt{29}$  &  $\sim 0.51$  \\
      $58$  &  $\sqrt{29}$  &  $\sim 0.50$  \\
      $59$  &  $4 + \sqrt{2}$ & $\sim 0.50$ \\
      $60$  &  $\sqrt{34}$  &  $\sim 0.57$  \\
      $61$  &  $1 + 2\,\sqrt{5}$  &  $\sim 0.49$  \\
      $62$  &  $4\,\sqrt{2}$  &  $\sim 0.52$  \\
      $63$  &  $\sqrt{34}$  &  $\sim 0.54$  \\
      $64$  &  $\sqrt{34}$  &  $\sim 0.53$  \\
      $65$  &  $\sqrt{34}$  &  $\sim 0.52$  \\
      $66$  &  $\sqrt{34}$  &  $\sim 0.52$  \\
      $67$  &  $\sqrt{10} + 2\,\sqrt{2}$  &  $\sim 0.54$  \\
    \end{tabular}\\[5mm] \hrule
    \begin{center}
      \caption{Maximal systolic ratios for origamis in $H(1,1)$ of given degree.}\label{H11}
      \end{center}
  \end{table}
\newpage
  
In Example~\ref{ExampleMax} we present an origami whose  systolic ratio is maximal among all origamis with less or equal than 67 squares.

\begin{ex}\label{ExampleMax}
  \textit{Figure~\ref{fig:maximalori}} shows an origami in $\mathcal{H}(1,1)$ with 30 squares that has systolic ratio $\frac{17}{30}$. This origami is given by the two permutations
  \[\begin{array}{lcl}
  \sigma_a &=& (1,2,3,4,5,6,7,8)(9,10,11,12,13,14,15,16,17,18,19,20,21,22,23,24,25,26,27,28,29,30),\\
  \sigma_b &=& (1,9,27,23,8,30,26,22,7,29,25,21,6,28,24,20,5,13,17,2,10,14,18,3,11,15,19,4,12,16).
  \end{array} 
  \]
  The dashed line is a systole of the surface.

\begin{figure}[h!]
\AffixLabels{%
\centerline{%
  \begin{xy}
<0.8cm,0cm>:
(0,3)*{\OriSquare{1}{}{a}{w}{}};
(1,3)*{\OriSquare{2}{}{}{}{}};
(2,3)*{\OriSquare{3}{}{}{}{}};
(3,3)*{\OriSquare{4}{}{}{}{}};
(4,3)*{\OriSquare{5}{}{}{}{}};
(5,3)*{\OriSquare{6}{}{f}{}{}};
(6,3)*{\OriSquare{7}{}{g}{}{}};
(7,3)*{\OriSquare{8}{w}{h}{}{}};
(7,0)*{\OriSquare{9}{z}{}{}{a}};
(1,4)*{\OriSquare{10}{}{}{z}{}};
(2,4)*{\OriSquare{11}{}{}{}{}};
(3,4)*{\OriSquare{12}{}{d}{}{}};
(4,4)*{\OriSquare{13}{v}{e}{}{}};
(1,5)*{\OriSquare{14}{}{b}{v}{}};
(2,5)*{\OriSquare{15}{u}{c}{}{}};
(0,2)*{\OriSquare{16}{}{}{u}{d}};
(1,2)*{\OriSquare{17}{}{}{}{e}};
(2,2)*{\OriSquare{18}{}{}{}{b}};
(3,2)*{\OriSquare{19}{}{}{}{c}};
(4,2)*{\OriSquare{20}{}{}{}{}};
(5,2)*{\OriSquare{21}{}{}{}{}};
(6,2)*{\OriSquare{22}{}{}{}{}};
(7,2)*{\OriSquare{23}{x}{}{}{}};
(4,1)*{\OriSquare{24}{}{}{x}{}};
(5,1)*{\OriSquare{25}{}{}{}{}};
(6,1)*{\OriSquare{26}{}{}{}{}};
(7,1)*{\OriSquare{27}{y}{}{}{}};
(4,0)*{\OriSquare{28}{}{}{y}{f}};
(5,0)*{\OriSquare{29}{}{}{}{g}};
(6,0)*{\OriSquare{30}{}{}{}{h}};
(3.5,-.47)*{{\bullet}};
(7.5,.53)*{{\bullet}};
(.5,4.53)*{{\bullet}};
(4.5,3.53)*{{\bullet}};
(3.5,1.53)*{{\circ}};
(2.5,5.53)*{{\circ}};
(-.5,2.53)*{{\circ}};
(7.5,2.53)*{{\circ}};
(3.55,-.4);(4.55,3.6) **@{.}
\end{xy} }}
  \caption{An origami $O$ with 30 squares whose maximal systole in the graph $\Gamma$ of saddle connections is $\sqrt{17}$.
    $O$ has two singularities $\bullet$ and $\circ$. The dashed line is a systole.}
\label{fig:maximalori}
\end{figure}
\end{ex}

We finish this section by constructing an example of an origami in an other stratum for which the length of the systoles in the graph of saddle connections is indeed shorter than the length of the systoles of the translation surface.


\begin{ex}\label{ex:Exception}
  We construct in the following an origami $O$ in the stratum $\HHH(1,1,1,1)$ such that its graph of saddle connections $\Gamma$ has a unique systole which corresponds to a null homotopic curve on the translation surface $S$ defined by $O$. Hence the systole of the translation surface is longer than the systole of the graph of saddle connections.\\
  
Let $O$ be an origami with one horizontal cylinder of height 1 and the gluing pattern shown in \textit{Figure~\ref{fig:ori_graph2}}. Thus we have the eight horizontal saddle connections $a$,  \ldots, $h$. Observe that we obtain four singularities $\blackcircle{1.1mm}$, $\whitecircle{1.1mm}$, $\mathsmaller{\blacksquare}$ and $\mathsmaller{\square}$, each of cone angle 4$\pi$. We choose the lengths of the segments $a$, \ldots, $h$ as follows:
  \[a = 1,\; b = 3,\; c = 3,\; d = 3,\; e = 7,\; f = 3,\; g = 3,\; h = 100.\]
  Thus $O$ is an origami in the stratum $\HHH(1,1,1,1)$  with 123 squares.  Its graph of horizontal saddle connections $\Gamma_{v}$ (with $v = {1 \choose 0}$) is shown on the left side of Figure~\ref{Figure:TheTwoGraphs}. In the following, one should keep in mind that the picture in Figure~\ref{fig:ori_graph2} is not true to scale. 
  \begin{center}
    \begin{figure}[h!]
      \begin{tikzpicture}[scale = .445]
        \tikzstyle{bullet}=[circle,thick,fill=black!,minimum size=2mm]
        \tikzstyle{place}=[circle,thick,draw=black!,fill=white!,minimum size=2mm]
        \tikzstyle{sq}=[rectangle,thick,draw=black!,fill=white!,minimum size=2mm]
        \tikzstyle{sqth}=[rectangle,thick,fill=black!,minimum size=2mm]
        \draw (0,0) -- node[left] {j}( 0,1);
        \draw (0,0) -- (35,0);
        \draw (35,0) -- (35,1);
        \draw (0,1) -- (35,1);
        \node at (1,1) (y1) {};
        \node at (8,1) (y2) {};
        \node at (11,1) (y3) {};
        \node at (14,1) (y4) {};
        \node at (17,1) (y5) {};
        \node at (20,1) (y6) {};
        \node at (32,1) (y7) {};
        \node at (35,1) (y8) {};
        \draw (0,1) -- node[above] {a} (y1);
        \draw (y1) -- node[above] {e} (y2);
        \draw (y2) -- node[above] {g} (y3);
        \draw (y3) -- node[above] {b} (y4);
        \draw (y4) -- node[above] {c} (y5);
        \draw (y5) -- node[above] {f} (y6);
        \draw (y6) -- node[above] {h} (y7);
        \draw (y7) -- node[above] {d} (y8);
        \}
        \node at (12,0) (x1) {};
        \node at (15,0) (x2) {};
        \node at (16,0) (x3) {};
        \node at (19,0) (x4) {};
        \node at (22,0) (x5) {};
        \node at (25,0) (x6) {};
        \node at (28,0) (x7) {};
        \node at (35,0) (x8) {};
        \draw (0,0) -- node[below] {h} (x1);
        \draw (x1) -- node[below] {b} (x2);
        \draw (x2) -- node[below] {a} (x3);
        \draw (x3) -- node[below] {f} (x4);
        \draw (x4) -- node[below] {g} (x5);
        \draw (x5) -- node[below] {d} (x6);
        \draw (x6) -- node[below] {c} (x7);
        \draw (x7) -- node[below] {e} (x8);
        \draw (0,0) -- node[below, right, distance=2in] {k} (1,1);
        \draw (x8) -- node[below, left, distance=5in] {l} (y7);
        \vertex[bullet] at (0,1) {};
        \vertex[bullet] at (x2) {};
        \vertex[bullet] at (y4) {};
        \vertex[bullet] at (x6) {};
        \vertex[bullet] at (y8) {};
        \vertex[place] at (y1) {};
        \vertex[place] at (x3) {};
        \vertex[place] at (y5) {};
        \vertex[place] at (x7) {};
        \vertex[sqth] at (y2) {};
        \vertex[sqth] at (x8) {};
        \vertex[sqth] at (0,0) {};
        \vertex[sqth] at (y6) {};
        \vertex[sqth] at (x4) {};
        \vertex[sq] at (y3) {};
        \vertex[sq] at (x5) {};
        \vertex[sq] at (y7) {};
        \vertex[sq] at (x1) {};
      \end{tikzpicture}      
      \caption{A surface $S$ defined by an origami $O$. Edges with same labels are glued.
        Furthermore, the two vertical edges are glued. $S$ has the four singularities $\blackcircle{1.1mm}$, $\whitecircle{1.1mm}$, $\mathsmaller{\blacksquare}$ and $\mathsmaller{\square}$.}
      \label{fig:ori_graph2}
    \end{figure}
  \end{center}
  For better notation we  identify the rectangle in \textit{Figure~\ref{fig:ori_graph2}} with the rectangle $[0,123] \times [0,1]$ and consider the points $x_0 = (0,0)$, $y_0 = (0,1)$ and $y_1 = (1,1)$. In addition to the horizontal saddle connection labeled by $a$ we obtain the vertical saddle connection $j$ which is the segment from $x_0$ to $y_0$ and the diagonal saddle connection $k$ which is the segment from $x_0$ to $y_1$. Observe that the concatenation of saddle connections $ak^{-}j$ is null homotopic and of length $r = 1 +1  + \sqrt{2} \sim 3.41$. It turns out that it defines the unique systole of the graph $\Gamma$. To see this, we compute using Algorithm I and II the subgraph $\Gamma_{\leq r}$ of $\Gamma$ containing only the edges of $\Gamma$ which correspond to saddle connections of length $\leq r$.
The resulting graph  $\Gamma_{\leq r}$ is shown on the right side of Figure~\ref{Figure:TheTwoGraphs} . The only additional saddle connection that we obtain is the saddle connection $l$ of length $\sqrt{10}$ in direction $v = {-3 \choose 1}$ which starts in the point $x_8 = (123,0)$ and ends in the point $y_7 = (120,1)$. We easily read off from $\Gamma_{\leq r}$ that its unique systole and thus also the unique systole of the full graph $\Gamma$ of saddle connections is  $ak^{-}j$.
The corresponding curve on the surface $S$ is null homotopic.
 \begin{figure}[h!]
  \[ 
          \begin{tikzpicture}[scale=1,<->,>=stealth',shorten >=1pt,auto,node distance=3cm,thick]
            \tikzstyle{place}=[circle,thick,draw=black!,minimum size=3mm]
            \tikzstyle{bullet}=[circle,thick,fill=black!,minimum size=3mm]
                  \tikzstyle{sq}=[rectangle,thick,minimum size=3mm]
        \tikzstyle{sqth}=[rectangle,thick,fill=black!,minimum size=3mm]
            \vertex[sq](p1) at (0,0) {};
            \vertex[sqth](p2) at (4,0) {};
            \vertex[bullet](p3) at (0,4) {};
            \vertex[place](p4) at (4,4) {};
            \path[-] (p1) edge node {g} (p2);
            \path[-] (p2)  edge   [bend left] node {h} (p1);
            \path[-] (p3) edge node {c} (p4);
            \path[-] (p3)  edge   [bend left] node {a} (p4);
            \path[-] (p1) edge node {b} (p3);
            \path[-] (p1)  edge   [bend left] node {d} (p3);
            \path[-] (p2) edge node {e} (p4);
            \path[-] (p2)  edge   [bend right] node {f} (p4);

          \end{tikzpicture}
          \hspace*{15mm}
        \begin{tikzpicture}[scale=1,<->,>=stealth',shorten >=1pt,auto,node distance=3cm,thick]
            \tikzstyle{place}=[circle,thick,draw=black!,minimum size=3mm]
            \tikzstyle{bullet}=[circle,thick,fill=black!,minimum size=3mm]
                  \tikzstyle{sq}=[rectangle,thick,minimum size=3mm]
        \tikzstyle{sqth}=[rectangle,thick,fill=black!,minimum size=3mm]
            \vertex[sq](p1) at (0,0) {};
            \vertex[sqth](p2) at (4,0) {};
            \vertex[bullet](p3) at (0,4) {};
            \vertex[place](p4) at (4,4) {};
            \path[-] (p2) edge [bend left] node {l,$\sqrt{10}$} (p1);
            \path[-] (p1) edge node {g,3} (p2);            
            \path[-] (p4) edge node {c,3} (p3);
            \path[-] (p3)  edge   [bend left] node {a,1} (p4);
            \path[-] (p4)  edge   [bend left = 90] node {k,$\sqrt{2}$} (p2);           
            \path[-] (p3) edge node {b,3} (p1);
            \path[-] (p1)  edge   [bend left] node {d,3} (p3);
            \path[-] (p2)  edge   [bend right] node {f,3} (p4);
             \path[-] (p2) edge node {j,1} (p3);

          \end{tikzpicture}
        \]
        \caption{Left side: Graph $\Gamma_v$ with $v = {1 \choose 0}$. \;\; Right side: Graph $\Gamma_{\leq r}$ with $r =  2 + \sqrt{2}$. The edges are labeled by the name and the length of the saddle connection.}
        \label{Figure:TheTwoGraphs}
 \end{figure}
\end{ex}

\vspace{2cm}

\noindent Tobias Columbus\\
\noindent Institute for Mathematics\\
\noindent AG Algebra and Number Theory\\
\noindent Paderborn University\\
\noindent 33098 Paderborn\\
\noindent Germany\\
\noindent  e-mail: \textit{columbus@math.uni-paderborn.de}\\

\noindent Frank Herrlich\\	
\noindent Institute of Algebra and Geometry\\ 
\noindent Karlsruhe Institute of Technology \\
\noindent 76128 Karlsruhe\\
\noindent Germany\\
\noindent e-mail: \textit{herrlich@kit.edu}\\
\\
\\
\noindent Bjoern Muetzel\\
\noindent Department of Mathematics\\
\noindent Eckerd College, St. Petersburg \\
\noindent Florida 33711 \\
\noindent USA\\
\noindent e-mail: \textit{bjorn.mutzel@gmail.com}\\
\\
\\
\noindent Gabriela Weitze-Schmith\"usen\\	
\noindent Department of Mathematics and Computer Science\\
\noindent Saarland University\\
\noindent 66123 Saarbr\"ucken\\
\noindent Germany\\
\noindent e-mail: \textit{weitze@math.uni-sb.de}\\

\end{document}